\documentclass[11pt]{amsart}
\calclayout
\usepackage{amsmath,amsfonts,amssymb,amsthm,epsfig,color}
\usepackage{graphicx}
\usepackage{subcaption}
\usepackage[english]{babel}
\usepackage[T1]{fontenc}
\usepackage{caption}
\usepackage{verbatim} 
\usepackage{enumerate}
\usepackage{tikz}
\usepackage{float}
\usepackage{bbm}
\usepackage{hyperref}
\usepackage[scr=esstix]{mathalpha}

\newtheorem*{thm*}{Theorem}
\newtheorem{thmglobal}{Theorem}
\newtheorem{thm}{Theorem}[section]

\newtheorem{lemma}[thm]{Lemma}
\newtheorem{prop}[thm]{Proposition}

\newtheorem{rem}[thm]{Remark}

\newtheorem{notation}[thm]{Notation}
\newtheorem{eg}[thm]{Example}

\theoremstyle{definition}
\newtheorem{defn}[thm]{Definition}

\newcommand{\R}{\mathbb{R}}

\newcommand{\Z}{\mathbb{Z}}

\newcommand{\id}{\mathrm{id}}
\newcommand{\Hom}{\mathrm{Hom}}

\newcommand{\ud}{\mathrm{d}}
\newcommand{\dd}[2]{\frac{\ud {#1}}{\ud {#2}}}

\newcommand{\1}[1]{#1^{(1)}}
\newcommand{\nn}[1]{#1^{({n})}}
\newcommand{\nk}[1]{#1^{(k)}}

\newcommand{\F}{{F}}
\newcommand{\K}{{K}}
\newcommand{\s}{{s}}

\title[Ergodic measures for periodic type $\Z^m$-skew-products over IETs]{Ergodic measures for periodic type $\Z^m$-skew-products over 
Interval Exchange Transformations}
\author{Yuriy Tumarkin}
\address{Institut f\"ur Mathematik, Universit\"at Z\"urich, Winterthurerstrasse 190, CH-8057 Z\"urich, Switzerland}
\email{yuriy.tumarkin@math.uzh.ch}

\begin{document}
\maketitle

\begin{abstract}
We consider a special case of the question of classification of invariant Radon measures of $\Z^m$-valued skew-products over interval exchange transformations, 
which arise as Poincaré sections of the linear flow on periodic infinite translation surfaces. In the case of periodic type skew-products, we 
obtain a full classification of ergodic invariant Radon measures, showing them to be precisely the Maharam measures, a family of measures 
parametrised by $\R^m$. For the proof we translate Rauzy-Veech renormalisation for skew-products into the symbolic language of the adic coding, and 
apply a symbolic result of Aaronson, Nakada, Sarig and Solomyak. Further, we use this language and a new extension of the Rauzy-Veech cocycle to 
find an explicit form for the Maharam measures and deduce the weak*-continuity of the measures depending on the parameter.

\end{abstract}

\section{Introduction}

For interval exchange transformations of finitely many intervals (which we will call finite IETs), and correspondingly linear flows on compact 
translation surfaces, the space of invariant measures is very well understood. 
Katok showed in \cite{Katok73} that for the linear flow on a genus $g$ surface, the cone of invariant measures has dimension at most $g$. 
In two celebrated papers, Masur (\cite{Masur82}) and Veech (\cite{Veech82}) showed that almost all IETs are uniquely ergodic with Lebesgue measure 
as the unique invariant measure. 
A few years later, using Masur's criterion (\cite{Masur92}), Kerckhoff, Masur and Smillie showed that on every 
translation surface, the linear flow in almost every direction is uniquely ergodic (\cite{KMS}). 

For IETs of infinitely many intervals (which we will call infinite IETs) and infinite translation surfaces however very little is still known. 
The simplest case of infinite IETs are $\Z^m$-valued skew-products over finite IETs. These arise as Poincar\'e sections of the linear flow on 
$\Z^m$-covers of compact translation surfaces. See \cite{HooperWeiss} for an introduction to $\Z$-covers of translation surfaces and $\Z$-valued 
skew-products, and the book 
in progress \cite{DHV} for a general introduction to infinite translation surfaces.\\

While for finite IETs one can consider Borel probability measures, for skew-products the invariant measures will be infinite.
The natural class of measures to consider in this setup are the \textbf{locally finite Borel measures}, also known as \textbf{Radon measures}, which 
are Borel measures that are finite on any compact set.

A large class of invariant measures for skew-products are the \textbf{Maharam measures}, first introduced by Maharam for $\R$-skew-products in 
\cite{Maharam}. For a Borel map $T:X\to X$, say a Borel measure $\nu$ is \textbf{quasi-invariant} if $\nu \circ T$ is absolutely continuous with 
respect to $\nu$. 
For a topological group $G$ and a map $\phi:X \to G$ consider the skew-product over $T$ with cocycle $\phi$,
\begin{equation}\label{eqn:skew-product}
\begin{split}
T_\phi: X\times G &\to X \times G\\
(x,a) &\mapsto (Tx,a+\phi(x)).
\end{split}
\end{equation}
Let $\psi: G \to \R$ be a continuous homomorphism.
If there exists a quasi-invariant measure $\nu_\psi$ for $T$ with 
\[\dd{\nu_\psi \circ T}{\nu_\psi} = e^{-\psi \circ \phi},\]
then one defines the Maharam measures $\mu_\psi$ on $X \times G$ by
\[\ud\mu_\psi(x,a) = \ud \nu_\psi(x) e^{\psi(a)} \ud \mu_G (a),\] 
where $\mu_G$ is the Haar measure on $G$. One checks by a direct calculation that $\mu_\psi$ is then invariant for $T_\phi$.

Note that for a compact group $G$ the only continuous homomorphism into $\R$ is constant, so the Maharam measures
arise from invariant measures for $T$. In our case $G = \Z^m$, so the set of continuous homomorphisms $G \to \R$ can be identified as $\R^m$.

\subsection{Known results}
The programme of studying measure rigidity of non-compact spaces was initiated in the paper \cite{ANSS} by Aaronson, Nakada, Sarig and Solomyak. 
A survey of this work was presented by Sarig in his 2010 ICM talk \cite{SarigICM}.

The classification of ergodic invariant Radon measures for skew-products over IETs is known only in a few cases.
\begin{enumerate}
	\item In \cite{ANSS} Aaronson, Nakada, Sarig and Solomyak consider \textbf{cylinder flows}, which are skew-products over rotations of the circle, 
	with skewing cocycle	
	\[\chi = (\beta+1)\mathbbm{1}_{[0,\frac{\beta}{\beta+1})} - \beta,\]
	for some $\beta > 0$. They show that the ergodic Radon measures are precisely the Maharam measures. 
	
	In addition in the second part of \cite{ANSS} the authors obtain an invariant measure classification for a class of symbolic systems called 
	\textbf{adic maps}. This is the main result that we apply in this article. 
	
	\item In \cite{PS} Pollicott and Sharp study invariant measures for
	the stable foliation of a certain class of pseudo-Anosov homeomorphisms of infinite abelian covers of surfaces. 
	Under several assumptions on the pseudo-Anosov, such as it being isotopic to the identity,  
	they use the symbolic result of \cite{ANSS} 
	to deduce that the ergodic invariant Radon measures for the foliation are precisely the Maharam measures.
	
	Our setup is closely related to this setting, since one can suspend a periodic type IET as the Poincaré section of the vertical flow on a 
	translation surface, so that the vertical foliation is the stable foliation of a pseudo-Anosov diffeomorphism on the surface.
	
	
	However one should note that the set of pseudo-Anosovs that we consider is disjoint from those covered by the result of \cite{PS}. Indeed, in our 
	setting the pseudo-Anosov always acts non-trivially on the homology of the surface, and hence is never isotopic to the identity.

	\item In \cite{HHW} Hooper, Hubert and Weiss deduce from the first result of \cite{ANSS} the classification of invariant Radon measures for the
	 linear flow on the infinite 
	staircase, and provide in this case a geometric interpretation for Maharam measures as pullbacks of Lebesgue measure on `deformed' staircases.
	
	\item In \cite{Hooper}, Hooper considers infinite translation surfaces with special cylinder decompositions, which arise from what is 
	known as the Thurston-Veech construction. For these surfaces, and for special `renormalisable' directions $\theta$, Hooper shows 
	that the ergodic measures for the linear flow $\varphi^\theta$ are precisely the Maharam measures, and 
	that these can be seen as pullbacks of Lebesgue measure on some other infinite translation surface with the same cylinder decomposition. 
	In particular this is a broad generalisation of the infinite staircase result of \cite{HHW}.

\end{enumerate}
One should note the analogy to the work of Babillot, Ledrappier and Sarig in \cite{BL} and \cite{Sarig04}, where they studied the ergodic invariant 
Radon measures for the horocycle flow on $\Z^m$-covers of compact hyperbolic surfaces. The ergodic measures they identify are analogous to Maharam 
measures in the sense that they scale under Deck transformations, with the scaling depending on a choice of homomorphism from $\Z^m$ to $\R$.
For an overview of this work, see Sarig's survey \cite{SarigICM}.\\

Let us also list some related results which answer similar but slightly different questions.

\begin{enumerate}
\item In \cite{ConzeFraczek} Conze and Fr\k{a}czek construct ergodic skew-products over periodic type IETs.
\item In \cite{HubertWeiss} Hubert and Weiss consider $\Z$-covers of Veech surfaces which admit infinite strips and prove ergodicity of the linear 
flow in almost every direction.
\item The work of Fr\k{a}czek and 
Ulcigrai (\cite{FU}) and later Fr\k{a}czek and Hubert (\cite{FH}) showed that for the generalised Ehrenfest (also known as wind-tree) model, under 
suitable conditions on the homology bundle and the Lyapunov exponents of the Kontsevich-Zorich cocycle, the linear flow in almost every direction 
is not ergodic with respect to Lebesgue measure.
\item M\'alaga and Troubetzkoy studied the ergodicity of linear flows on non-periodic infinite translation surfaces, proving ergodicity 
in almost every direction for the generic aperiodic Ehrenfest model in \cite{MalagaTroubetzkoyEhrenfest}, and unique ergodicity in almost every 
direction for the generic aperiodic infinite staircase surface in \cite{MalagaTroubetzkoyStaircase}.

\end{enumerate}

\subsection{Our results} In this paper we treat the case of skew-products which are periodic for Rauzy-Veech renormalisation (see section \ref{sec:iets}
 for a precise definition 
of Rauzy-Veech renormalisation for skew-products over IETs). We obtain the same classification as the other results, namely that the
ergodic invariant Radon measures are exactly the Maharam ones. 

The precise statement of this result is:

\begin{thmglobal}\label{thm:main}
	Suppose {$T:K\to K$} is an infinitely complete IET of $d$ intervals, $\phi \in G^d$ for $G = \Z^m$.
	If the skew-product $T_\phi: \K \times G \to \K \times G$ is of periodic type (see definition \ref{def:periodic})
	 then the ergodic invariant Radon measures	 for $T_\phi$ are precisely the Maharam measures. To be precise:
	\begin{enumerate}[(I)]
		\item For every homomorphism $\psi: G \to \R$ there exists a unique (up to scaling) measure $\nu_\psi$ on $\K$ which is quasi-invariant under $T$ with 
		$\dd{\nu_\psi \circ T}{\nu_\psi} = e^{-\psi \circ \phi}$, and this measure is non-atomic. 
		
		Given such a $\nu_\psi$ we define the Maharam measure $\mu_\psi$ on $\K\times G$, with 
		\[\ud\mu_\psi(x,a) = \ud \nu_\psi(x) e^{\psi(a)} \ud \mu_G (a),\] where $\mu_G$ is the Haar measure on $G$.
		
		\item Each Maharam measure $\mu_\psi$ is an ergodic invariant Radon measure for $T_\phi$.
		\item Every ergodic invariant Radon measure for $T_\phi$ is up to scale a Maharam measure for some homomorphism $\psi$.
	\end{enumerate}
\end{thmglobal}

Further, we obtain an exact expression for the Maharam measures of certain {sets} arising from dynamical partitions and 
deduce:

\begin{thmglobal}\label{thm:continuity}
In the same setting as above, the measures $\mu_\psi$ depend weak-* continuously on $\psi$.
\end{thmglobal}

The proof of Theorem \ref{thm:main}, similarly to the result proved by Pollicott and Sharp in \cite{PS} exploits symbolic coding and relies 
on the result from \cite{ANSS} in the symbolic setting.
We want to stress that while \cite{PS} uses the classical Markov coding for a pseudo-Anosov,
one of the main contributions of this paper is in developing the symbolic
renormalisation dictionary for skew-products using Rauzy-Veech induction and the adic coding induced by it,
in the spirit of Bufetov (\cite{Bufetov}) and Lindsey-Trevi\~no (\cite{LT}). 
We emphasise that this point of view extends to all skew-products over IETs rather than just periodic type ones. 
While currently, to the best of our knowledge, the symbolic result of \cite{ANSS} has not yet been proven in the setting of non-stationary 
Markov chains, if it were to be proven in this non-stationary setting, this point of view could lead to an analogous measure 
classification for a much broader class of skew-products by a similar argument.

Another contribution of this paper is in introducing the \emph{level-counting cocycle}, an extension of the Rauzy-Veech cocycle that we believe 
will be a useful tool for studying skew-products over IETs. As an example application we give the proof of Theorem \ref{thm:continuity}.

\subsection{Structure of the paper} The outline of the paper is as follows:\\

\noindent Section \ref{sec:philosophy} is a brief and informal description of the philosophy of the proof of Theorem \ref{thm:main}.
\noindent Sections \ref{sec:iets} and \ref{sec:Bratteli} are preliminaries.
\begin{itemize}
\item In section \ref{sec:iets} we give the most important basic definitions for IETs and skew-products.
\item In section \ref{sec:Bratteli} we explain Bratteli diagrams and how they can be used to code IETs, following Bufetov (\cite{Bufetov}) and Lindsey-Trevi\~no
(\cite{LT}).
\end{itemize} 
Sections \ref{sec:symb}-\ref{sec:weak*} deal with periodic type skew-products.
\begin{itemize}
\item In section \ref{sec:symb} we explain how for periodic type IETs one can see renormalisation symbolically.
\item In section \ref{sec:ANSS} we apply the results of \cite{ANSS} to prove Theorem \ref{thm:main}.
\item In section \ref{sec:weak*} we introduce the level-counting cocycle and use it to give an explicit 
formula for Maharam measures and deduce Theorem \ref{thm:continuity}.
\end{itemize} 

\subsection{Acknowledgements}
I would like to thank my advisor Corinna Ulcigrai for introducing me to this subject and suggesting this problem,
and for her continuous support and guidance. I would also like to thank Albert Artiles for suggesting the question of weak-* continuity of Maharam 
measures. I would like to thank the referee for some very helpful comments on the first draft of this paper. 
This work was supported by the Swiss National Science Foundation through Grant 200021\_188617/1 and the Zurich Graduate School in 
Mathematics. 

\section{Idea of the proof of Theorem \ref{thm:main}}\label{sec:philosophy}
Here we briefly present the way we apply the symbolic result of \cite{ANSS} in the proof of Theorem \ref{thm:main}.

\subsection{The ANSS result}\label{subsec:ideaANSS}
The full statement of the result of \cite{ANSS} is given in section \ref{sec:ANSS}, but we give an abbreviated version here to convey the main idea.

Given a one-sided subshift of finite type $\sigma: \Sigma \to \Sigma$, one can consider the \textbf{tail equivalence relation}
\[\mathfrak{T}(\sigma) := \{(x,y) \in \Sigma^2\ |\  \exists n \geq 0: \sigma^n x = \sigma^n y\}.\]
In certain cases one can define a map $\tau:\Sigma \to \Sigma$ called the \textbf{adic map} (see section \ref{sec:Bratteli} for the definition) 
for which its \textbf{orbit equivalence relation}
\[\mathcal{O}(\tau) := \{(x,y) \in \Sigma^2\ |\  \exists n \in \Z: \tau^n x = y\}\]
satisfies
\[\mathcal{O}(\tau) = \mathfrak{T}(\sigma).\]

Heuristically one can think of this relationship meaning ``$\sigma$ renormalises $\tau$''. A prototype for this is the geodesic flow $g_t$ 
renormalising the horocycle flow $h_s$ on the unit tangent bundle of a hyperbolic surface $S$. In this case, as the flow has continuous time, the tail 
equivalence relation for $g_t$ identifies vectors whose orbits converge in the future, rather than agree after a finite number of steps as in 
the definition for a discrete time map. Precisely,
\[\mathfrak{T}(g_t) := \{(\xi,\xi') \in (T^1 S)^2: \lim_{t\to +\infty} d(g_t(\xi),g_t(\xi')) = 0\}.\]
The definition of a horocycle is exactly what guarantees $\mathcal{O}(h_s) = \mathfrak{T}(g_t)$.\\

Now given a map $f:\Sigma \to G$ one can consider the skew-product $\sigma_f: \Sigma\times G \to \Sigma\times G$. We know that $\sigma$ renormalises 
$\tau$, but what does $\sigma_f$ renormalise? One can answer this question as follows:

If one defines the \textbf{tail cocycle} $\phi_f: \Sigma \to G$,
\[\phi_f(x) := \sum_{i=0}^{\infty} \big(f(\sigma^i x ) - f(\sigma^i \tau x )\big),\]
then
\[\mathcal{O}(\tau_{\phi_f}) = \mathfrak{T}(\sigma_f).\]
Here $\tau_{\phi_f}:\Sigma \times G \to \Sigma \times G$ denotes the skew-product over $\tau$ with cocycle $\phi_f$ (c.f. \ref{eqn:skew-product}).
This equality of relations is explained in \cite{ANSS}, but for completeness we also give the proof of the equality in Appendix \ref{app:orb=tail}.\\

The theorem of \cite{ANSS} states that if $f$ is finite memory and aperiodic (see section \ref{sec:ANSS} for definitions), then the ergodic invariant 
Radon measures for the skew-product $\tau_{\phi_f}$ are exactly the Maharam measures.

\subsection{How to apply it}
We will show in section \ref{sec:Bratteli} that we can code an IET $T$ as an adic map $\tau$ on some shift space $\Sigma$, and the skew-product 
$T_\phi$ as the skew-product $\tau_{\phi}$ over the adic map.\\

The key step in the proof of Theorem \ref{thm:main} is in the construction of a cocycle $f:\Sigma \to G$ for which the skew-product $\sigma_f$ satisfies
\[\mathcal{O}(\tau_{\phi}) = \mathfrak{T}(\sigma_f),\]
from which it follows that $\mathcal{O}(\tau_{\phi_f}) = \mathcal{O}(\tau_{\phi})$. Hence the invariant sets, and so the ergodic measures are the same for 
$\tau_{\phi_f}$ and  $\tau_{\phi}$. So if we apply the theorem from \cite{ANSS} to $\tau_{\phi_f}$, we deduce that the ergodic invariant 
Radon measures for $\tau_\phi$ are precisely the Maharam measures, as desired.

We construct $f$ in section \ref{sec:symb} using the dictionary of symbolic renormalisation. In section \ref{sec:ANSS} we check that $f$ satisfies 
the required conditions and complete the proof.

\section{Background and definitions}\label{sec:iets}
\subsection{Interval Exchange Transformations}
\begin{defn} 
An \textbf{Interval Exchange Transformation} (IET for short) of $d$ intervals is a map $\F:[a,b)\to [a,b) $, 
which is a {right-continuous} piecewise 
orientation-preserving isometry, with all discontinuities contained in a set of $d-1$ marked points $\{a < x_1 <\dots < x_{d-1} < b\}$. 
(Note: for the purpose of this paper we allow $\F$ to be continuous at these marked points.) 
Let $x_0 = a, x_d = b$ and let $I_i$ be the interval $[x_{i-1},x_i)$ for $1\leq i\leq d$.

$\F$ satisfies \textbf{Keane's condition} if for any $0\leq i \leq d$, any $n\geq 1$, $\F^n(x_i)$ does not belong to the set of marked points 
$\{x_1,\dots,x_{d-1}\}$. Let $\mathcal{X}_d$ be the space of IETs of $d$ intervals satisfying Keane's condition. Let $\mathcal{X}^0_d \subset \mathcal{X}_d$ be 
the subspace of IETs defined on the unit interval $I = [0,1)$.
\end{defn}

\subsubsection{Compactification}\label{subsubsec:compactification}
For technical reasons, it is convenient to modify the definition of IETs somewhat so that an IET becomes a continuous mapping of a compact space. 
Here we follow the approach of \cite{MMY}, compactifying the interval to a Cantor set by doubling the orbits of singularities into left and right
copies. We remark that this idea first appeared in Keane's work in \cite{Keane75}.

More precisely, let $F \in \mathcal{X}^0_d$. Let $0 < x_1 < \dots x_{d-1} < 1$ be the marked points containing the set of discontinuities of $F$. 
Let $y_1 < \dots < y_{d-1}$ be the points $\{f(x_i): 0 \leq i \leq d-1\} \setminus \{0\}$.  For $n\geq 0$ let
\begin{align*}
D_n^+ &= \{F^{n}(y_i):1\leq i\leq d-1\},\\
D_n^- &= \{F^{-n}(x_i):1\leq i\leq d-1\}.
\end{align*}

As $F$ is Keane, these sets are disjoint from each other, and none of them contain $0$.
Let $D$ be the union $D = \bigsqcup_{n\geq 0} \big(D_n^+ \sqcup D_n^-\big)$.

Define the atomic measure $\rho$ by
\[\rho = \sum_{n\geq 0}\ \sum_{x\in D_n^+ \sqcup D_n^-} 2^{-n} \delta_x.\]

Define the increasing maps $l: [0,1] \to \R$ and $r:[0,1) \to \R$ by
\begin{align*}
l(x) &= \rho([0,x))\\
r(x) &= \rho([0,x]).
\end{align*}

Then $l(x) = r(x)$ for all $x\in I \setminus D$, whereas for $x\in D^\pm_n$, $r(x) = l(x) + 2^{-n}$. For $x\in D$, $l(x)$ and $r(x)$ are the split left 
and right copies of $x$.
Further, if $x<x'$, then $r(x) < l(x')$, where the inequality is strict by the minimality of $F$.

Finally, define 
\[K = l((0,1]) \cup r([0,1)) \subset \R.\]
Then $K$ is a Cantor set, with gaps $(l(x),r(x))$ for $x\in D$.\\

Note that the maps $l,r$ and the Cantor set $K$ all depend on the IET $F$ we started with. This dependence will always be taken to be implicit 
in this paper.

\begin{lemma}
There exists a homeomorphism $T: K\to K$ such that $T \circ l = l \circ F$.
\end{lemma}
\begin{proof}
For $x\in (0,1)$, define $T(l(x)) = l(F(x))$ and $T(r(x)) = r(F(x))$. Define also $T(0) = T(r(0)) = r(F(0))$ and $T(l(1)) = l(\lim_{x\to 1^-} F(x))$. 

The map $T$ is well-defined, as for $x,x' \in (0,1)$, if $l(x) = r(x')$, then $x=x'$. Then $x\notin D$, 
and hence $F(x) \notin D$, meaning that $l(F(x)) = r(F(x))$. To see that $T$ is continuous, note 
that if $|x-x'| < 1$, then there is no $l(y)$ inbetween $x$ and $x'$ for $y\in D^-_0$, meaning that $F$ is continuous on 
$l^{-1}([x,x'])$, and hence $|T(x) - T(x')| \leq 2 |x-x'|$.

Similarly one shows that $T^{-1}$ is well-defined and continuous.
\end{proof}

Let $m_K$ be the measure on $K$ defined as the pushforward by $l$ of Lebesgue measure on $I$. Then $m_K$ is non-atomic and invariant by $T$.

From now on we will also refer to $T: K\to K$ as an IET, and this map will be the main focus of our study.

\subsubsection{Renormalisation for IETs}
A key tool for studying IETs is the Rauzy-Veech renormalisation $\mathcal{R}$. 
We do not give an explicit definition of Rauzy-Veech renormalisation here, only recalling some basic properties, and refer the reader to the many excellent 
introductions available in the literature, such as \cite{Viana}, \cite{Yoccoz}, \cite{Zorich}.

The \textbf{Rauzy-Veech induction} is a map $\hat{\mathcal{R}}:\mathcal{X}_d \to \mathcal{X}_d$. Given $\F:[a,b)\to [a,b)$,   
the Rauzy-Veech induction procedure
produces a sequence of shrinking nested intervals $[a,b)=I^{(0)} \supset I^{(1)} \supset I^{(2)} \supset \dots$, which all have the same left endpoint. 
The map $\hat{\mathcal{R}}^k(\F)$ is defined as the induced map of $\F$ on $I^{(k)}$, also denoted by $\F^{(k)}:I^{(k)}\to I^{(k)}$. 
The subintervals of $I^{(k)}$ permuted by $\F^{(k)}$ are labelled by $I^{(k)}_i$ for $1 \leq i \leq d$, where the order of these subintervals depends on $\F$.

The \textbf{Rauzy-Veech renormalisation} is a map $\mathcal{R}:\mathcal{X}^0_d \to \mathcal{X}^0_d$, defined by 
$\mathcal{R}(\F) = \frac{\1\F}{|\1I|}$.

For the continuous extension $T:K \to K$ of $F:I \to I$, if $I^{(k)} = [0,b_k)$, we define $K^{(k)} = l\big((0,b_k]\big) \cup r\big([0,b_k)\big)$. 
We define also the induced map $T^{(k)}: K^{(k)} \to K^{(k)}$, which satisfies $T^{(k)} \circ l = l \circ F^{(k)}$. Similarly define the subsets 
$K^{(k)}_i$ as the images of the intervals $I^{(k)}_i$, including the right copy of the left endpoint and the left copy of the right endpoint. This makes 
each $K_i^{(k)}$ a clopen subset of $K$.

\begin{defn}
For $0\leq r < k$ let $q^{(r,k)}_j$ be the return time of $T^{(r)}$ to $\K^{(k)}$ on $\K^{(k)}_j$, i.e. 
\[q^{(r,k)}_j = \min\Big\{n \geq 1 : \big(T^{(r)}\big)^n \Big( \K^{(k)}_j\Big) \subset \K^{(k)}\Big\}.\]

The \textbf{Rokhlin tower} $Z^{(r,k)}_j$ is the union 
\[Z^{(r,k)}_j = \bigsqcup_{l=0}^{q^{(r,k)}_j-1} \big(T^{(r)}\big)^l \left( \K^{(k)}_j \right).\]

Then by definition of the induced map and $q^{(r,k)}_j$, $\K^{(r)} = \bigsqcup_{j=1}^d Z^{(r,k)}_j$.

For $0 \leq l < q^{(r,k)}_j$ we call $\big(T^{(r)}\big)^l \left( \K^{(k)}_j \right)$ the $l^\text{th}$ \textbf{floor} of the tower $Z^{(r,k)}_j$, and $\K^{(k)}_j$ the \textbf{base}.
\end{defn}

\begin{figure}[H]
\centering
\begin{tikzpicture}[scale=7.5]
	\draw[red,line width=1] (0,0) -- (0.3,0);
	\draw[red,line width=1] (0.4,0) -- (0.6,0);
	\draw[red,line width=1] (0.7,0) -- (1.2,0);
	
	\draw[line width=1] (0,0.1) -- (0.3,0.1);
	\draw[line width=1] (0.4,0.1) -- (0.6,0.1);
	\draw[line width=1] (0.7,0.1) -- (1.2,0.1);
	
	\draw[line width=1] (0,0.2) -- (0.3,0.2);
	\draw[line width=1] (0.4,0.2) -- (0.6,0.2);
	\draw[line width=1] (0.7,0.2) -- (1.2,0.2);
	
	\draw[line width=1] (0,0.3) -- (0.3,0.3);
	\draw[line width=1] (0.4,0.3) -- (0.6,0.3);
	\draw[line width=1] (0.7,0.3) -- (1.2,0.3);
	
	\draw[line width=1] (0.4,0.4) -- (0.6,0.4);
	\draw[line width=1] (0.7,0.4) -- (1.2,0.4);
	
	\draw[line width=1] (0.4,0.5) -- (0.6,0.5);
	
	\draw[line width=1] (0.4,0.6) -- (0.6,0.6);
	
	\draw[->] (1.22,0.01) arc (-80:80:0.04);
	\draw[->] (1.22,0.11) arc (-80:80:0.04);
	\draw[->] (1.22,0.21) arc (-80:80:0.04);
	\draw[->] (1.22,0.31) arc (-80:80:0.04);
	
	\node at (1.34,0.05) {$T^{(r)}$};
	\node at (1.34,0.15) {$T^{(r)}$};
	\node at (1.34,0.25) {$T^{(r)}$};
	\node at (1.34,0.35) {$T^{(r)}$};
	
	\draw[->] (1.22,0.41) arc (80:-150:0.6 and 0.3);
	\node at (1.64,0.1) {$T^{(r)}$};
		
	\node[red] at (0.15,-0.05) {\small $\K^{(k)}_{1}$};
	\node[red] at (0.5,-0.06) {\small $\K^{(k)}_{2}$};
	\node[red] at (0.95,-0.05) {\small $\K^{(k)}_{3}$};
	
	\node at (0.95,0.46) {\small $\Big(T^{(r)}\Big)^{q^{(r,k)}_3-1} \Big(\K^{(k)}_{3}\Big)$};
	
	\draw[<->,dashed] (-0.07,0) -- (-0.07,0.3);
	\node[text width = 1.8 cm] at (-0.2,0.15) {$\#$ of floors \\ $= q_1^{(r,k)}$}; 
	
	\draw[blue] (0.38,-0.02) -- (0.62,-0.02) -- (0.62,0.62) -- (0.38,0.62) -- cycle;
	\node[blue] at (0.3,0.53) {$Z^{(r,k)}_2$};
	
\end{tikzpicture}
\caption{An example of towers for an IET on 3 intervals. The union of the bases is $\K^{(k)}$, and the union of all the floors  (including 
the base) is $\K^{(r)}$. From the towers we can not see exactly the image under $T^{(r)}$ of the top floor of a tower, but we do know that it gets 
sent into $\K^{(k)}$, hence somewhere in the union of the bases. }
\end{figure}
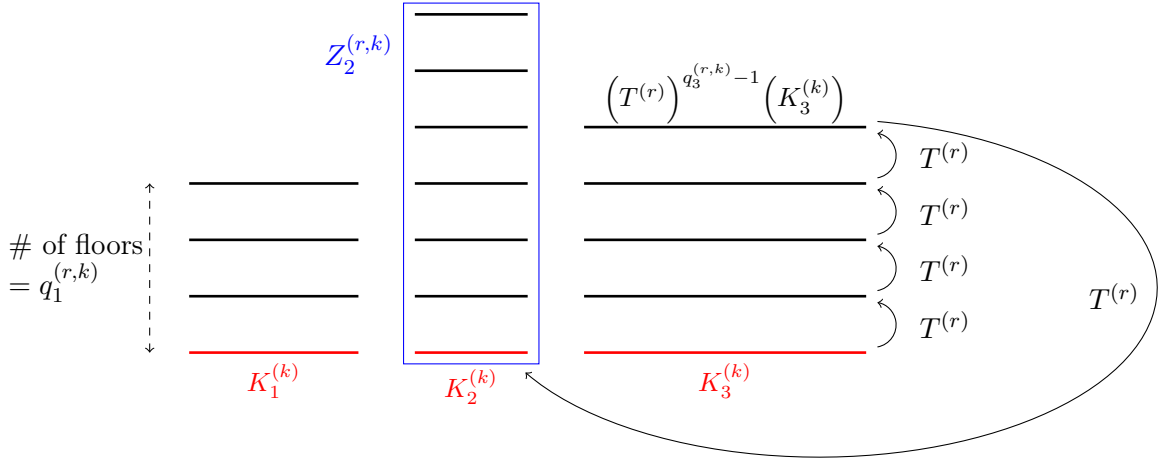

\begin{defn}
The \textbf{Rauzy-Veech cocycle} is defined as integer-valued $d\times d$ matrices $A^{(r,k)}$, with
\[A^{(r,k)}_{ij} = \#\left\{0\leq l < q^{(r,k)}_j: \big(T^{(r)}\big)^l \left(\K^{(k)}_j\right) \subset \K^{(r)}_i\right\}.\]
Thus $A^{(r,k)}_{ij}$ counts the number of floors in the tower $Z^{(r,k)}_j$ which are subsets of $\K^{(r)}_i$.\\

It is a cocycle since for $r<s<k$, $A^{(r,k)}_{ij} = \sum_{t=1}^d A^{(r,s)}_{it} A^{(s,k)}_{tj}$, and so $A^{(r,k)} = A^{(r,s)}A^{(s,k)}$.

\end{defn}

\begin{defn}
$T$ is \textbf{infinitely complete} if for any $i,j$,
\[ \lim_{n \to \infty} A^{(0,n)}_{ij} \to \infty.\]
\end{defn}

If $\mathcal{R}^N(T) = T$, then for any $k\geq 1 $, 
\[A^{(0,kN)}_{ij} = \big(A^{(0,N)}\big)^k_{ij},\]
hence $T$ is infinitely complete if $A^{(0,N)}$ is a positive matrix.

\begin{rem}
In fact for a finite IET $T$, if $T$ satisfies Keane's condition then $T$ is automatically infinitely complete. (See \cite{Yoccoz}.) 
\end{rem}

\subsection{Skew-products over IETs}
\begin{notation}
For a dynamical system $T:X \to X$, and a function $f$ on $X$, we will denote by $S^T_n f (x)$ the Birkhoff sum
\[S^T_n f(x) = \sum_{i=0}^{n-1} f(T^ix).\]
\end{notation}
\begin{defn}
Let $\F\in \mathcal{X}^0_d$ be a Keane IET of $d$ intervals. 
Let $G = \Z^m$ for some $m$, and let $\phi: I \to G$ be a function that is constant on $I_i$ for each $i$.
Then the map $\F_\phi: I \times G \to I \times G$ given by
\[ \F_\phi(x,a) = (\F(x),a+\phi(x))\]
is the \textbf{skew-product} over $\F$ with skewing cocycle $\phi$.
\end{defn}

We call $\phi$ a cocycle because $\F_\phi^n(x,a) = \big(\F^n(x),a+S^{\F}_n \phi (x)\big)$, where the Birkhoff sum $S^{\F}_n \phi (x)$ 
acts as a dynamical cocycle.

Without loss of generality we will always assume that the image of $\phi$ generates $G$ 
(otherwise the span of the image is isomorphic to $\Z^{m'}$ for some $m'$ and we can take $G$ to be this group). 
We will also treat $\phi$ as a $G$-valued $d$-vector, with $\phi_j = \phi\big|_{I_j}$.

{Given a $\phi: I \to G$ as above, define $\hat \phi: K \to G$ by $\hat \phi|_{K_i} = \phi_i$. 
Then for the extension $T:K \to K$, its skew-product $T_\phi: K \times G \to K\times G$ is defined analogously by 
$T_\phi(x,a) = (T(x),a+\hat \phi(x))$. From now on we will denote $\hat \phi$ also by $\phi: K \to G$.}

\subsubsection{Rauzy-Veech renormalisation for skew-products}
Let $\F_\phi:I\times G\to I\times G$ be a skew-product as above.

\begin{defn}If $\hat{\mathcal{R}}(\F)$ is the induced map of $\F$ on $\1I$, then the \textbf{Rauzy-Veech induction} $\hat{\widetilde{\mathcal{R}}}(\F_\phi)$ 
for $\F_\phi$ is defined by inducing on $\1I \times G$. We also write $(\F_\phi)^{(k)}$ to mean $\Big(\hat{\widetilde{\mathcal{R}}}\Big)^k (\F_\phi)$ .

The \textbf{Rauzy-Veech renormalisation} $\widetilde{\mathcal{R}}(\F_\phi)$ is defined by scaling 
back up to unit length in the first coordinate, i.e. if $\hat{\widetilde{\mathcal{R}}}(\F_\phi) (x,a) = (y,b)$, then
\[\widetilde{\mathcal{R}}(\F_\phi) (x,a) = \left(\frac{y}{|\1I|},b \right),\]
so that $\widetilde{\mathcal{R}}(\F_\phi)$ is a map from $I\times G$ to itself.
\end{defn}

Let us write down the relation between the Rauzy-Veech renormalisation for skew-products with the renormalisation for the base IET. We will need the 
following standard lemma expressing Birkhoff sums at special times via the action of the renormalisation matrix.
\begin{lemma}\label{lem:BS}
Let $\F$ be an IET, $\phi$ a $G$-valued $d$-vector as above. If $\nn\phi=(A^{(0,{n})})^T\phi$, 
then for $x\in \nn I_j$, $S^{\F}_{q^{(0,{n})}_j} \phi (x) = \nn\phi(x)$.
\end{lemma}
\begin{proof}
Since $\phi(\F^l(x)) = \phi_i$ for $i$ such that $I_i \ni \F^l(x)$,
\begin{align*}
S^{\F}_{q^{(0,{n})}_j} \phi (x)&= \sum_{l=0}^{q^{(0,{n})}_j-1} \phi(\F^l(x))\\
&= \sum_{i=1}^d \phi_i \cdot \#\{0\leq l < q^{(0,{n})}_j: \F^l(\nn I_j) \subset I_i\}\\
&= \sum_{i=1}^d A^{(0,{n})}_{ij} \phi_i\\
&=\big((A^{(0,{n})})^{\F}\phi\big)_j = \nn\phi_j = \nn\phi(x).\qedhere
\end{align*}
\end{proof}

\begin{prop}\label{height cocycle}
Let $\F_\phi$ be a skew-product as above. If $\nn\phi=(A^{(0,{n})})^T\phi$, then
\[\hat{\widetilde{\mathcal{R}}} (\F_\phi) = \nn\F_{\nn\phi}\] and hence \[\widetilde{\mathcal{R}} (\F_\phi) = (\mathcal{R}\F)_{\nn\phi}.\]
\end{prop}
\begin{proof}

Suppose $(x,a)\in \nn I_j \times G$. For any $k$, 
$\F_\phi^k(x,a) = \big(\F^k(x),a+S^{\F}_k\phi (x)\big)$. Hence the return time of $(x,a)$ to $\nn I\times G$ under $\F_\phi$ is the same as the return time of $x$ to 
$\nn I$ under $\nn\F$, which is $q^{(0,{n})}_j$. By definition of the induced map, $\nn\F(x) = \F^{q^{(0,{n})}_j}(x)$.\\

Hence, applying Lemma \ref{lem:BS},
\begin{align*}
\hat{\widetilde{\mathcal{R}}} (\F_\phi) (x,a) &= \F_\phi^{q^{(0,{n})}_j}(x,a)\\
&= \left( \F^{q^{(0,{n})}_j}(x), a+ S^{\F}_{q^{(0,{n})}_j} \phi (x) \right)\\
&= \left(\nn\F(x),a+ S^{\F}_{q^{(0,{n})}_j} \phi (x) \right)\\
&= \left(\nn\F(x),a+ \nn \phi (x) \right)\\
&= \nn\F_{\nn\phi} (x,a).
\end{align*}
Hence $\hat{\widetilde{\mathcal{R}}} (\F_\phi) = \nn\F_{\nn\phi}$, and rescaling in the $x$ coordinate we see that 
$\widetilde{\mathcal{R}} (\F_\phi) = (\mathcal{R}\F)_{\nn\phi}$.
\end{proof}

{For the continuous extension $T_\phi: K \times G \to K \times G$, we denote by $\big(T_\phi\big)^{(n)}$ the induced map on $K^{(n)}\times G$. 
By Proposition \ref{height cocycle}, $\big(T_\phi\big)^{(n)}$ is equal to the skew-product $T^{(n)}_{\nn\phi}$.}

\begin{notation}
From now on we will always denote $(A^{(0,n)})^T\phi$ by $\nn \phi$.
\end{notation}
\begin{defn}
\label{def:periodic}
A skew-product $\F_\phi$, where $\F$ is an infinitely complete IET, is said to be of \textbf{periodic type} if it is a periodic point for 
$\widetilde{\mathcal{R}}$, i.e. if for some $N>0$, $\widetilde{\mathcal{R}}^N (\F_\phi) = \F_\phi$. 
(Cf. the notion of periodic type for IETs first defined in \cite{SinaiUlcigrai}.) {The continuous extension 
$T_\phi:K \to K$ of $F_\phi: I \to I$ is said to be of \textbf{periodic type} if $F_\phi$ is of periodic type.}
\end{defn}

\begin{rem}
It is a standard fact that for every loop $\gamma$ in a Rauzy diagram, if the Rauzy-Veech matrix $A$ corresponding to $\gamma$ is positive, then 
the Perron-Frobenius eigenvector of $A$ gives length data that defines a periodic type IET. If in addition $A$ has an eigenvalue equal to 1, then 
taking $\phi \in \Z^d$ to be a left eigenvector with eigenvalue 1, the skew-product $T_\phi$ is of periodic type. To have a $\Z^m$- skew-product of 
periodic type one requires an at least $m$-dimensional 1-eigenspace for $A^T$, hence these are rarer.
\end{rem}

\subsubsection{Towers for skew-products}

As with finite IETs, we can define towers for skew-products. Denote by $\K^{(k)}_{j,a}$ the set $\K^{(k)}_j \times \{a\} \subset \K \times G$.

Note that since $(T_\phi)^{(r)}$ is a skew-product over $T^{(r)}$, the return time of $(T_\phi)^{(r)}$ to $\K^{(k)}\times G$ on $\K^{(k)}_{j,a}$ is equal to $q^{(r,k)}_j$ independently of $a$.

Let ${\tilde Z}^{(r,k)}_{j,a}$ be the Rokhlin tower for $(T_\phi)^{(r)}$ with base $\K^{(k)}_{j,a}$, i.e.
\[{\tilde Z}^{(r,k)}_{j,a} = \bigsqcup_{l=0}^{q^{(r,k)}_j-1} \big((T_\phi)^{(r)}\big)^l \Big( \K^{(k)}_{j,a} \Big).\]

Since $T_\phi$, and hence $(T_\phi)^{(r)}$ commutes with translation in the $G$ component, we see that the towers ${\tilde Z}^{(r,k)}_{j,a}$ 
for different $a$ 
vary only by translation in the $G$ component. In fact, when projected onto the $\K$ component, each of these towers looks like $Z^{(r,k)}_j$, the tower 
for the original IET $T$.

As in the finite case, the Rokhlin towers give a dynamical partition

\[\K^{(r)} \times G = \bigsqcup_{a\in G} \bigsqcup_{j=1}^d {\tilde Z}^{(r,k)}_{j,a}.\]

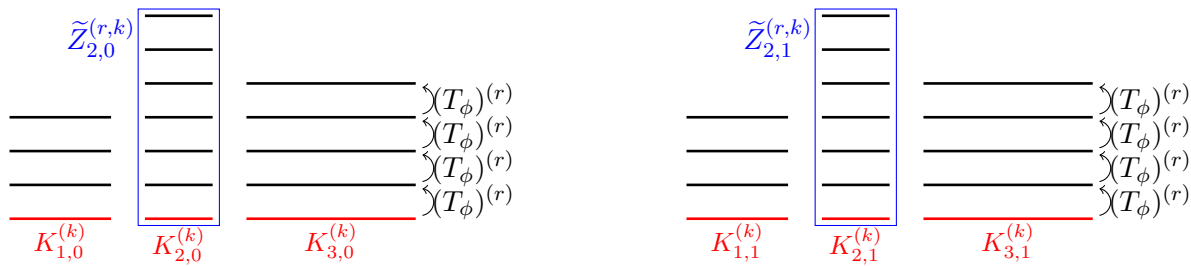
\begin{figure}[H]
\centering
\begin{tikzpicture}[scale=4.5]
	\draw[red,line width=1] (0,0) -- (0.3,0);
	\draw[red,line width=1] (0.4,0) -- (0.6,0);
	\draw[red,line width=1] (0.7,0) -- (1.2,0);
	
	\draw[line width=1] (0,0.1) -- (0.3,0.1);
	\draw[line width=1] (0.4,0.1) -- (0.6,0.1);
	\draw[line width=1] (0.7,0.1) -- (1.2,0.1);
	
	\draw[line width=1] (0,0.2) -- (0.3,0.2);
	\draw[line width=1] (0.4,0.2) -- (0.6,0.2);
	\draw[line width=1] (0.7,0.2) -- (1.2,0.2);
	
	\draw[line width=1] (0,0.3) -- (0.3,0.3);
	\draw[line width=1] (0.4,0.3) -- (0.6,0.3);
	\draw[line width=1] (0.7,0.3) -- (1.2,0.3);
	
	\draw[line width=1] (0.4,0.4) -- (0.6,0.4);
	\draw[line width=1] (0.7,0.4) -- (1.2,0.4);
	
	\draw[line width=1] (0.4,0.5) -- (0.6,0.5);
	
	\draw[line width=1] (0.4,0.6) -- (0.6,0.6);
	
	\draw[->] (1.22,0.01) arc (-80:80:0.04);
	\draw[->] (1.22,0.11) arc (-80:80:0.04);
	\draw[->] (1.22,0.21) arc (-80:80:0.04);
	\draw[->] (1.22,0.31) arc (-80:80:0.04);
	
	\node at (1.37,0.05) {$(T_\phi)^{(r)}$};
	\node at (1.37,0.15) {$(T_\phi)^{(r)}$};
	\node at (1.37,0.25) {$(T_\phi)^{(r)}$};
	\node at (1.37,0.35) {$(T_\phi)^{(r)}$};
		
	\node[red] at (0.15,-0.07) {\small $\K^{(k)}_{1,0}$};
	\node[red] at (0.5,-0.08) {\small $\K^{(k)}_{2,0}$};
	\node[red] at (0.95,-0.07) {\small $\K^{(k)}_{3,0}$};

	\draw[blue] (0.38,-0.02) -- (0.62,-0.02) -- (0.62,0.62) -- (0.38,0.62) -- cycle;
	\node[blue] at (0.27,0.53) {$\widetilde{Z}^{(r,k)}_{2,0}$};
	
	\draw[red,line width=1] (2,0) -- (2.3,0);
	\draw[red,line width=1] (2.4,0) -- (2.6,0);
	\draw[red,line width=1] (2.7,0) -- (3.2,0);
	
	\draw[line width=1] (2,0.1) -- (2.3,0.1);
	\draw[line width=1] (2.4,0.1) -- (2.6,0.1);
	\draw[line width=1] (2.7,0.1) -- (3.2,0.1);
	
	\draw[line width=1] (2,0.2) -- (2.3,0.2);
	\draw[line width=1] (2.4,0.2) -- (2.6,0.2);
	\draw[line width=1] (2.7,0.2) -- (3.2,0.2);
	
	\draw[line width=1] (2,0.3) -- (2.3,0.3);
	\draw[line width=1] (2.4,0.3) -- (2.6,0.3);
	\draw[line width=1] (2.7,0.3) -- (3.2,0.3);
	
	\draw[line width=1] (2.4,0.4) -- (2.6,0.4);
	\draw[line width=1] (2.7,0.4) -- (3.2,0.4);
	
	\draw[line width=1] (2.4,0.5) -- (2.6,0.5);
	
	\draw[line width=1] (2.4,0.6) -- (2.6,0.6);
	
	\draw[->] (3.22,0.01) arc (-80:80:0.04);
	\draw[->] (3.22,0.11) arc (-80:80:0.04);
	\draw[->] (3.22,0.21) arc (-80:80:0.04);
	\draw[->] (3.22,0.31) arc (-80:80:0.04);
	
	\node at (3.37,0.05) {$(T_\phi)^{(r)}$};
	\node at (3.37,0.15) {$(T_\phi)^{(r)}$};
	\node at (3.37,0.25) {$(T_\phi)^{(r)}$};
	\node at (3.37,0.35) {$(T_\phi)^{(r)}$};
		
	\node[red] at (2.15,-0.07) {\small $\K^{(k)}_{1,1}$};
	\node[red] at (2.5,-0.08) {\small $\K^{(k)}_{2,1}$};
	\node[red] at (2.95,-0.07) {\small $\K^{(k)}_{3,1}$};

	\draw[blue] (2.38,-0.02) -- (2.62,-0.02) -- (2.62,0.62) -- (2.38,0.62) -- cycle;
	\node[blue] at (2.27,0.53) {$\widetilde{Z}^{(r,k)}_{2,1}$};
	
\end{tikzpicture}
\caption{A few of the infinitely many towers for a skew-product over an IET on 3 intervals, here with $G = \Z$.
The union of the bases is $\K^{(k)} \times G$, and the union of all the floors  (including the base) is $\K^{(r)} \times G$.}
\end{figure}

\newpage
\section{Bratteli diagrams and adic maps}\label{sec:Bratteli}
In this section we define the coding for IETs that we will use.
This section is based on the expositions in \cite{Bufetov} and \cite{LT}.
\subsection{Bratteli diagrams}

\begin{defn}
A \textbf{Bratteli diagram} is an infinite directed graph with vertex set $V$ and edge set $\mathcal{E}$, which are partitioned into finite sets 
$V = \bigsqcup_{k\geq 0} V_k$ and $\mathcal{E} = \bigsqcup_{k\geq 1} \mathcal{E}_k$, such that if $s,t:\mathcal{E} \to V$ are the source and target maps, then
$s(\mathcal{E}_k) = V_{k-1}$ and $t(\mathcal{E}_k) = V_k$.\\
This means that the edges in $\mathcal{E}_k$ go from a vertex in $V_{k-1}$ to a vertex in $V_k$, but 
also that for all $k$ each vertex in $V_k$ is the source for at least one edge, and for $k>0$ also the target of at least one edge.

An infinite sequence of edges $(e_1,e_2,\dots) \in \Pi_{k\geq 1} \mathcal{E}_k$ is called \textbf{admissible} if for every $k\geq 1$, $t(e_k) = s(e_{k+1})$. 
Let $\Sigma$ be the set of all admissible sequences, which correspond to infinite paths in the Bratteli diagram.

Similarly define $\Sigma_k$ as the space of finite admissible paths from $V_0$ to $V_k$. For such a path $p = (e_1,\dots,e_k)\in \Sigma_k$
let $s(p) = s(e_1)$ and $t(p) = t(e_k)$.
\end{defn}

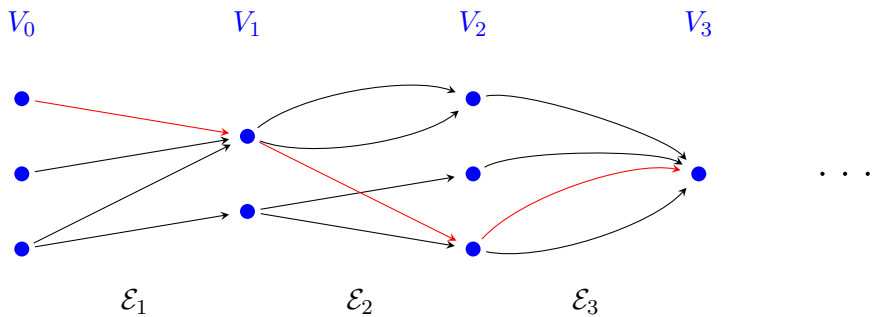
\begin{figure}[H]
\centering
\begin{tikzpicture}[scale=1,->,>=stealth,shorten > = 4pt,shorten < = 2pt]
	\node[circle, fill=blue, inner sep = 2pt] at (0,0) (00) {};
	\node[circle, fill=blue, inner sep = 2pt] at (0,1) (01) {};
	\node[circle, fill=blue, inner sep = 2pt] at (0,2) (02) {};
	
	\node[circle, fill=blue, inner sep = 2pt] at (3,0.5) (10) {};
	\node[circle, fill=blue, inner sep = 2pt] at (3,1.5) (11) {};
	
	\node[circle, fill=blue, inner sep = 2pt] at (6,0) (20) {};
	\node[circle, fill=blue, inner sep = 2pt] at (6,1) (21) {};
	\node[circle, fill=blue, inner sep = 2pt] at (6,2) (22) {};
	
	\node[circle, fill=blue, inner sep = 2pt] at (9,1) (3) {};
	
	\path
	(00) edge (11)
	(01) edge (11)
	(02) edge[red] (11)
	(00) edge (10);
	
	\path 
	(10) edge (20)
	(10) edge (21)
	(11) edge[red] (20)
	(11) edge[bend left,looseness=0.7] (22)
	(11) edge[bend right,looseness=0.7] (22);
	
	\path 
	(22) edge[bend left,looseness=0.5] (3)
	(21) edge[bend left,looseness=0.5] (3)
	(20) edge[bend left,looseness=0.7,red] (3)
	(20) edge[bend right,looseness=0.7] (3);
	
	\node[] at (11,1) {\huge\dots};
	
	\node[blue] at (0,3) {$V_0$};
	\node[blue] at (3,3) {$V_1$};
	\node[blue] at (6,3) {$V_2$};
	\node[blue] at (9,3) {$V_3$};
	
	\node[] at (1.5,-0.7) {$\mathcal{E}_1$}; 
	\node[] at (4.5,-0.7) {$\mathcal{E}_2$}; 
	\node[] at (7.5,-0.7) {$\mathcal{E}_3$};

\end{tikzpicture}
\caption{An example of a Bratteli diagram, with a path in $\Sigma_3$ highlighted in red.}
\end{figure}

Without loss of generality we assume that for each $k$, $|\mathcal{E}_k| >1$. Indeed, as we are interested in studying paths in the graph, were there only a 
single edge between $V_{k-1}$ and $V_k$, we would lose no information by collapsing it. 

\begin{defn}
A \textbf{partially ordered Bratteli diagram} is a Bratteli diagram together with a partial order $<$ on $\mathcal{E}$, under which two edges are comparable 
if and only if they have the same target. 

Such a partial order induces a \textbf{lexicographic partial order} on $\Sigma$, with two infinite paths 
$p_1=(e_1,e_2,\dots)$ and $p_2=(f_1,f_2,\dots)$ comparable under $<$ if and only if they have the same tail, meaning that for some 
$K$, for all $k > K, e_k = f_k$. 
If $p_1$ and $p_2$ have the same tail, and $K$ is minimal, i.e. $e_K \neq f_K$, then we say $p_1 < p_2$ if and only if $e_K < f_K$.

An edge $e_k\in \mathcal{E}_k$ is said to be maximal under $<$ if for all other edges $e'$ that are comparable to it, $e > e'$.
An infinite path $p = (e_1,e_2,\dots)$ is said to be \textbf{maximal} under $<$ if every edge $e_k$ of the path is maximal. These are precisely the paths that 
are maximal under the lexicographic order.
Let $\Sigma_{max} \subset \Sigma$ be the space of all maximal paths.

Similarly an edge is said to be \textbf{minimal} if it is minimal under the lexicographic order, and a path is minimal if all its edges are minimal. 
Let $\Sigma_{min} \subset \Sigma$ be the space of all minimal paths.

\end{defn}

\begin{defn}
The \textbf{adic map} (also known as Vershik map) $\tau: \Sigma\setminus \Sigma_{max} \to \Sigma$ is the map which sends a path $p$ to the smallest path $p'$ 
for which $p' > p$, i.e.

\[\tau(p) = \min\{q\in \Sigma: q > p\}.\]
\end{defn}

\begin{prop}
The minimum in the definition of the adic map is always attained, hence $\tau$ is well-defined.
\end{prop}
\begin{proof}
Let $p = (e_1,e_2,\dots)\in \Sigma\setminus\Sigma_{max}$.
Suppose that $q = (f_1,f_2,\dots)\in \Sigma$ is comparable with $p$. Then there is some minimal $K$, such that for all $k>K$, $e_k=f_k$. 
Say in this case that $p$ and $q$ 
agree after $K$. By definition of lexicographic order, if $p$ and $q$ agree after $K$, then $p >q$ if and only if $e_K>f_K$, and conversely.

As $p$ is not maximal, there exists a smallest $n$ such that $e_n$ is not maximal. There are finitely many paths that agree with $p$ after $n$, of which $p$ 
can not be the maximal under $<$ (since there exists some edge $e\in \mathcal{E}_n$ with $e > e_n$, and any path $q = (f_1,\dots,f_{n-1},e,e_{n+1},\dots)$ 
satisfies $q>p$). Let $p'$ be the minimal of these paths that agree with $p$ after $n$ and that are greater than $p$ under $<$.

Any path $q = (f_1,f_2,\dots)$ which agrees with $p$ after $m$ for some $m<n$ has $q<p$, since $e_m$ is maximal, so $f_m < e_m$. 

If a path $q = (f_1,f_2,\dots)$ agrees with $p$ after $m$ for some $m>n$ and $q>p$, then $f_m > e_m$. Then $p'$ also agrees with $q$ after $m$, 
and so $q>p'$.

Thus $p'$ is the minimal path greater than $p$ under $<$, hence $\tau(p) = p'$.
\end{proof}

In fact, if for $p = (e_1,\dots)$, $n$ is minimal such that $e_n$ is not a maximal edge, then $\tau(p) = (f_1,f_2,\dots, f_n,e_{n+1},\dots)$, where:
\begin{itemize}
\item $f_n$ is the minimal edge greater than $e_n$,
\item For $k<n$, $f_k = \min\{e\in \mathcal{E}_k: t(e) = s(f_{k+1})\}$.
\end{itemize}

\begin{eg}
Consider the Bratteli diagram with $V_k$ consisting of a single vertex for each $k$, and $\mathcal{E}_k$ having two edges for each $k$, one labelled $0_k$ 
and the other labelled $1_k$. Let $<$ be the partial order $0_k < 1_k$ for all $k$. 

The set of infinite paths in this Bratteli diagram corresponds to 2-adic integers. Indeed, we can represent any 2-adic integer
\[x = \sum_{i=1}^{\infty} \xi_i 2^i\]
by the path 
\[p(x) = (\xi_1,\xi_2,\xi_3,\dots) \in \Sigma.\]

Then $\tau$ corresponds to the addition of 1,
\[\tau(p(x)) = p(x+1).\]
In this case $\tau$ is also known as the 2-adic odometer map.
\bigskip\\

\begin{figure}[H]
\centering
\begin{tikzpicture}[scale=1,->,>=stealth,shorten > = 4pt,shorten < = 2pt, node distance = 2cm, vertex/.style={circle,inner sep=2pt,fill=blue}]
	\node[vertex] (0) {};
	\node[vertex] (1) [right of=0] {};
	\node[vertex] (2) [right of=1] {};
	\node[vertex] (3) [right of=2] {};
	\node[vertex] (4) [right of=3] {};
	\node[vertex] (5) [right of=4] {};
	\node[vertex] (6) [right of=5] {};
	
	\path
	(0) edge[bend left,red] node[above]{1} (1)
	(0) edge[bend right] node[below]{0} (1);
	
	\path
	(1) edge[bend left] node[above]{1} (2)
	(1) edge[bend right,red] node[below]{0} (2);
	
	\path
	(2) edge[bend left,red] node[above]{1} (3)
	(2) edge[bend right] node[below]{0} (3);
	
	\path
	(3) edge[bend left] node[above]{1} (4)
	(3) edge[bend right,red] node[below]{0} (4);
	
	\path
	(4) edge[bend left] node[above]{1} (5)
	(4) edge[bend right,red] node[below]{0} (5);
	
	\path
	(5) edge[bend left] node[above]{1} (6)
	(5) edge[bend right,red] node[below]{0} (6);
	
	\node[] [right of=6] {\huge\dots};

\end{tikzpicture}
\caption{The Bratteli diagram corresponding to 2-adic integers. The path highlighted in red, continued with all 0s, corresponds to 
$2^0 + 2^2 = 5$.}
\end{figure}
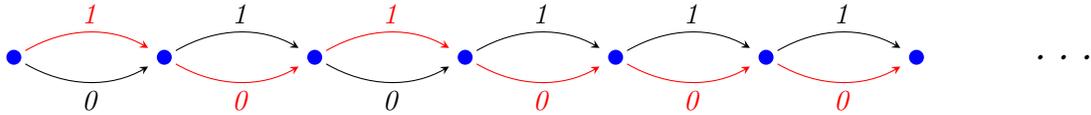

\end{eg}

\newpage
\subsection{Coding IETs by adic maps}
\subsubsection{Bratteli diagrams for IETs}
Fix an infinitely complete $d$-IET $T$.
Let $(n_k)$ be an increasing sequence, let $J^{(k)}=\K^{(n_k)}$, and $J^{(k)}_j=\K^{(n_k)}_j$ for $j=1,\dots,d$.   
Assume that $(n_k)$ is chosen such that for any $k$, 
$A^{(n_k,n_{k+1})}$ is a positive matrix (it suffices for $n_{k+1}-n_k$ to be large enough by infinite-completeness). 

Let $T_k = T^{(n_k)}$ be the induced map on $J^{(k)}$. 
For all $j,k$ let $Z^{(k)}_j = Z^{(n_{k-1},n_k)}_j$ be the tower for $T_{k-1}$ on $J^{(k)}_j$.
 Let $q^{(k)}_j = q^{(n_{k-1},n_k)}_j$ be the height of this tower. 

\begin{defn}
We define a Bratteli diagram for $T$ corresponding to the sequence $(n_k)$ as follows:
\begin{itemize}
\item Vertex sets: $V_k = \{1,...,d\}$ for $k\geq 0$.
\item The edge set $\mathcal{E}_k$ corresponds to the set of floors of the towers $Z^{(k)}_j$ for $j=1,\dots,d$. Formally:
\[\mathcal{E}_k = \{(j,l):1\leq j \leq d, 0 \leq l < q^{(k)}_j\}.\]
Here we should interpret the edge $(j,l)\in\mathcal{E}_k$ as corresponding to the floor $T_{k-1}^l(J^{(k)}_j)$.
\item For $e = (j,l)\in \mathcal{E}_k$, we define $t(e)=j$. We define $s(e) = i$, such that $T_{k-1}^l(J^{(k)}_j) \subset J^{(k-1)}_i$.
\item The partial order: two edges $(j,l)$ and $(j',l')$ in $\mathcal{E}_k$ are comparable if and only if their targets match, i.e. if and only if $j=j'$. 
(Meaning the edges correspond to floors in the same tower). 
Then we say the edge corresponding to the higher floor in the tower is greater, i.e.
\[(j,l) > (j,l') \iff l>l'.\]

\end{itemize}
\end{defn}

We describe now how finite paths in this Bratteli diagram correspond to floors in Rokhlin towers for $T$.

For $p_k = ((j_1,l_1),(j_2,l_2),\dots)\in \Sigma_k$, let

\[J(p_k) = T^{l_1} \circ T_1^{l_2} \circ \dots \circ T_{k-1}^{l_k} \Big(J^{(k)}_{j_k}\Big).\]

Informally, we can think of $T,T_1,\dots,T_{k-1}$ as `elevators' of different speed that go up the towers $Z^{(k)}_j$. $T$ is the slowest elevator which stops 
at every floor, whereas $T_1$ only stops on floors that are subsets of $\1J$, $T_2$ is even faster and stops more rarely, and $T_{k-1}$ is the express 
elevator that only stops on the floors that are parts of $J^{(k-1)}$.

To get from the base $J^{(k)}_j$ to some specific floor $F$ in the fastest way, one first takes the express elevator to 
the last floor belonging to $J^{(k-1)}$ which is below $F$, then one needs to switch to the slower elevator $T_{k-2}$ and take it
to its last stop before $F$, and so on until one arrives at $F$ by taking $T$.

If we denote by $l_r$ the number of stops we have taken $T_{r-1}$ for, then
\[F = T^{l_1} \circ T_1^{l_2} \circ \dots \circ T_{k-1}^{l_k} \Big(J^{(k)}_{j}\Big).\]

Defining $p_k = ((j_1,l_1),\dots,(j_k,l_k))$ with $j_k = j$ and $j_{r-1} = s((j_r,l_r))$ for $r<k$, we see that $F = J(p_k)$.

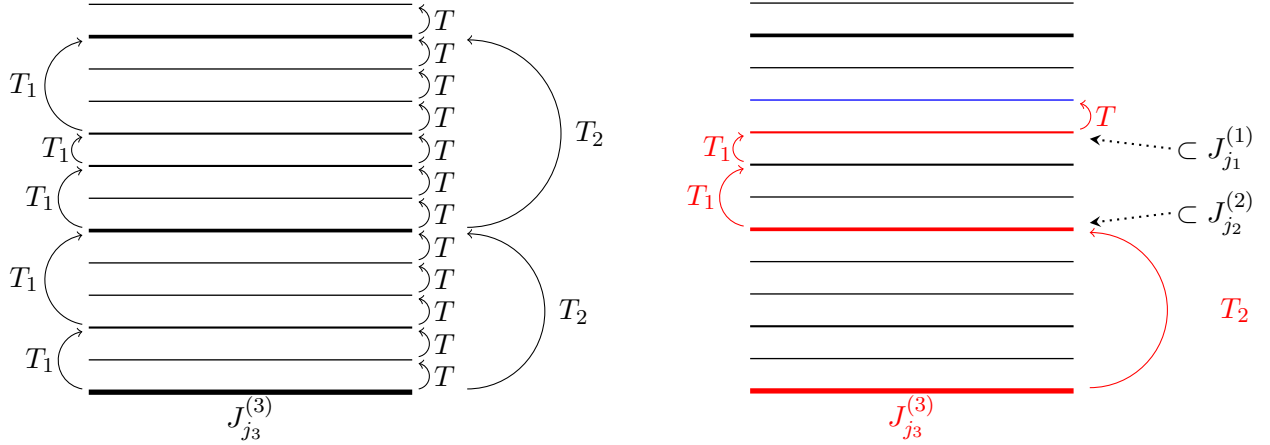
\begin{figure}[H]
\centering
\begin{minipage}{0.45\textwidth}
\begin{tikzpicture}[scale=4.3]
	\draw[line width=2] (0,0) -- (1,0);
	\draw[line width=0.5] (0,0.1) -- (1,0.1);
	\draw[line width=0.8] (0,0.2) -- (1,0.2);
	\draw[line width=0.5] (0,0.3) -- (1,0.3);
	\draw[line width=0.5] (0,0.4) -- (1,0.4);
	\draw[line width=1.4] (0,0.5) -- (1,0.5);
	\draw[line width=0.5] (0,0.6) -- (1,0.6);
	\draw[line width=0.8] (0,0.7) -- (1,0.7);
	\draw[line width=0.8] (0,0.8) -- (1,0.8);
	\draw[line width=0.5] (0,0.9) -- (1,0.9);
	\draw[line width=0.5] (0,1) -- (1,1);
	\draw[line width=1.4] (0,1.1) -- (1,1.1);
	\draw[line width=0.5] (0,1.2) -- (1,1.2);

	\draw[->] (1.02,0.01) arc (-80:80:0.04);
	\draw[->] (1.02,0.11) arc (-80:80:0.04);
	\draw[->] (1.02,0.21) arc (-80:80:0.04);
	\draw[->] (1.02,0.31) arc (-80:80:0.04);
	\draw[->] (1.02,0.41) arc (-80:80:0.04);
	\draw[->] (1.02,0.51) arc (-80:80:0.04);
	\draw[->] (1.02,0.61) arc (-80:80:0.04);
	\draw[->] (1.02,0.71) arc (-80:80:0.04);
	\draw[->] (1.02,0.81) arc (-80:80:0.04);
	\draw[->] (1.02,0.91) arc (-80:80:0.04);
	\draw[->] (1.02,1.01) arc (-80:80:0.04);
	\draw[->] (1.02,1.11) arc (-80:80:0.04);
	
	\node at (1.1,0.05) {$T$};
	\node at (1.1,0.15) {$T$};
	\node at (1.1,0.25) {$T$};
	\node at (1.1,0.35) {$T$};
	\node at (1.1,0.45) {$T$};
	\node at (1.1,0.55) {$T$};
	\node at (1.1,0.65) {$T$};
	\node at (1.1,0.75) {$T$};
	\node at (1.1,0.85) {$T$};
	\node at (1.1,0.95) {$T$};
	\node at (1.1,1.05) {$T$};
	\node at (1.1,1.15) {$T$};

	\draw[->] (-0.02,0.01) arc (260:100:0.09);
	\draw[->] (-0.02,0.21) arc (260:100:0.14);
	\draw[->] (-0.02,0.51) arc (260:100:0.09);
	\draw[->] (-0.02,0.71) arc (260:100:0.04);
	\draw[->] (-0.02,0.81) arc (260:100:0.14);
	
	\node at (-0.15,0.1) {$T_1$};
	\node at (-0.2,0.35) {$T_1$};
	\node at (-0.15,0.6) {$T_1$};
	\node at (-0.1,0.75) {$T_1$};
	\node at (-0.2,0.95) {$T_1$};
	
	\draw[->] (1.17,0.01) arc (-90:90:0.24);
	\node at (1.5,0.25) {$T_2$};
	\draw[->] (1.17,0.51) arc (-90:90:0.29);
	\node at (1.55,0.8) {$T_2$};
	
	\node at (0.5,-0.08) {$J^{(3)}_{j_3}$};
\end{tikzpicture}
\end{minipage}
\hspace{0.08\textwidth}
\begin{minipage}{0.45\textwidth}
\begin{tikzpicture}[scale=4.3]
	\draw[line width=2,red] (0,0) -- (1,0);
	\draw[line width=0.5] (0,0.1) -- (1,0.1);
	\draw[line width=0.8] (0,0.2) -- (1,0.2);
	\draw[line width=0.5] (0,0.3) -- (1,0.3);
	\draw[line width=0.5] (0,0.4) -- (1,0.4);
	\draw[line width=1.4,red] (0,0.5) -- (1,0.5);
	\draw[line width=0.5] (0,0.6) -- (1,0.6);
	\draw[line width=0.8] (0,0.7) -- (1,0.7);
	\draw[line width=0.8,red] (0,0.8) -- (1,0.8);
	\draw[line width=0.5,blue] (0,0.9) -- (1,0.9);
	\draw[line width=0.5] (0,1) -- (1,1);
	\draw[line width=1.4] (0,1.1) -- (1,1.1);
	\draw[line width=0.5] (0,1.2) -- (1,1.2);
	
	\draw[->,red] (1.02,0.81) arc (-80:80:0.04);

	\node[red] at (1.1,0.85) {$T$};

	\draw[->,red] (-0.02,0.51) arc (260:100:0.09);
	\draw[->,red] (-0.02,0.71) arc (260:100:0.04);
	
	\node[red] at (-0.15,0.6) {$T_1$};
	\node[red] at (-0.1,0.75) {$T_1$};
	
	\draw[->,red] (1.05,0.01) arc (-90:90:0.24);
	\node[red] at (1.5,0.25) {$T_2$};
	
	\node[red] at (0.5,-0.08) {$J^{(3)}_{j_3}$};

	\draw[->,>=stealth,dotted,line width = 0.8] (1.3,0.55) -- (1.05,0.52);
	\draw[->,>=stealth,dotted,line width = 0.8] (1.3,0.75) -- (1.05,0.78);
	
	\node at (1.44,0.55) {$\subset J^{(2)}_{j_2}$};
	\node at (1.44,0.75) {$\subset J^{(1)}_{j_1}$};

\end{tikzpicture}
\end{minipage}
\caption{The elevator on a tower $Z^{(0,3)}_{j_3}$. The blue floor is $T^1\circ T_1^2 \circ T_2^1 \Big( J^{(3)}_{j_3}\Big)$. 
Hence it is $J((j_1,1),(j_2,2),(j_3,1))$ for appropriate $j_1,j_2$.}
\end{figure}

\begin{lemma}\label{paths are floors}
For any $p_k\in \Sigma_k$, $J(p_k)$ is a floor in the tower $Z^{(0,n_k)}_{t(p_k)}$.

Conversely, for every floor $F$ of the tower $Z^{(0,n_k)}_j$, there exists a unique $p_k \in \Sigma_k$ such that $J(p_k) = F$.
\end{lemma}

\begin{proof}
Let $p_k = (e_1,\dots,e_k)\in \Sigma_k$, with $p_m = (j_m,l_m)$. Then $J(p_k) = T^r \Big(J^{(k)}_{j_k}\Big)$ for some $r$. 
It is a floor of $Z^{(0,n_k)}_{j_k}$ if for $0 < n \leq r$, $T^n \Big(J^{(k)}_{j_k}\Big) \not\subset J^{(k)}$.

For $0 \leq m \leq k$, let $r_m$ be such that 
\[F_m= T^{r_m} \Big(J^{(k)}_{j_k}\Big)\ ,\  \text{where}\ \ F_m = T_{m}^{l_{m+1}} \circ \dots \circ T_{k-1}^{l_k} \Big(J^{(k)}_{j_k}\Big).\]
This gives a decreasing sequence $0< r_k < r_{k-1} < \dots < r_0=r$.
\medskip

Firstly, by (backwards) induction on $m$, $F_m \subset J^{(m)}_{j_m}$ for all $m$. Indeed, this is clear for $m=k$, and by definition of $s(e_m)$, 
\[T_{m-1}^{l_m}\Big(J^{(m)}_{j_m}\Big) \subset J^{(m-1)}_{s(e_m)}.\] 
Hence if $F_m \subset J^{(m)}_{j_m}$, then 
\[F_{m-1} = T_{m-1}^{l_m} (F_m) \subset T_{m-1}^{l_m}\Big(J^{(m)}_{j_m}\Big) \subset  J^{(m-1)}_{s(e_m)} = J^{(m-1)}_{j_{m-1}},\]
where $s(e_m) = t(e_{m-1}) = j_{m-1}$ follows from the admissibility of $p_k$.

Now suppose for contradiction, that for some $0 < n \leq r$, $T^n \Big(J^{(k)}_{j_k}\Big) \subset J^{(k)}$. Let $m < k$ be such that 
$r_{m+1} < n \leq r_m $. Then for some $s_1 \geq q^{(m+1)}_{j_{m+1}}$, \[T^n \Big(J^{(k)}_{j_k}\Big) = T_m^{s_1} (F_{m+1}),\] 
and for some $s_2 \geq 0$, \[F_m = T_m^{s_2} \left( T^n \Big(J^{(k)}_{j_k}\Big) \right).\] 
Hence as $F_m = T_m^{l_{m+1}} (F_{m+1})$, \[l_{m+1} = s_1 + s_2 \geq q^{(m+1)}_{j_{m+1}}.\]
But this is impossible by definition of $\mathcal{E}_{m+1}$.

Thus we have shown that $J(p_k)$ is a floor of $Z^{(0,n_k)}_{t(p_k)}$.\\

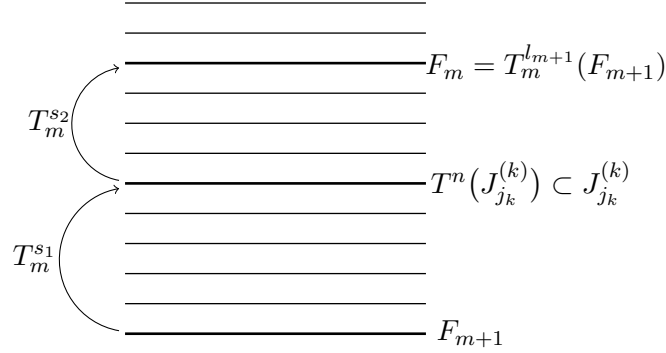
\begin{figure}[H]
\centering
\begin{tikzpicture}[scale=4]
	\draw[line width=1] (0,0) -- (1,0);
	\draw[line width=0.5] (0,0.1) -- (1,0.1);
	\draw[line width=0.5] (0,0.2) -- (1,0.2);
	\draw[line width=0.5] (0,0.3) -- (1,0.3);
	\draw[line width=0.5] (0,0.4) -- (1,0.4);
	\draw[line width=1] (0,0.5) -- (1,0.5);
	\draw[line width=0.5] (0,0.6) -- (1,0.6);
	\draw[line width=0.5] (0,0.7) -- (1,0.7);
	\draw[line width=0.5] (0,0.8) -- (1,0.8);
	\draw[line width=1] (0,0.9) -- (1,0.9);
	\draw[line width=0.5] (0,1) -- (1,1);
	\draw[line width=0.5] (0,1.1) -- (1,1.1);
	
	\node[] at (1.15,0) {$F_{m+1} $};
	\node[] at (1.35,0.5) {$T^n\big(J^{(k)}_{j_k}\big) \subset J^{(k)}_{j_k}$};
	\node[] at (1.4,0.9) {$F_{m} = T_m^{l_{m+1}}(F_{m+1})$};
	
	\draw[->] (-0.02,0.01) arc (260:100:0.24);  
	\draw[->] (-0.02,0.51) arc (260:100:0.19);  
	
	\node[] at (-0.3,0.25) {$T_m^{s_1}$};
	\node[] at (-0.25,0.7) {$T_m^{s_2}$};

\end{tikzpicture}
\caption{$J(p_k)$ is a floor of $Z^{(0,n_k)}_j$ - proof by contradiction.}
\end{figure}

\bigskip

Conversely, let $F = T^r \Big(J^{(k)}_j\Big)$ be a floor of $Z^{(0,n_k)}_j$ for some $j$. We will find $p_k \in \Sigma_k$ such that $F = J(p_k)$. 

Define for $0 \leq m \leq k$
\[r_m = \max\{0 \leq n\leq r: T^n\Big(J^{(k)}_j\Big) \subset J^{(m)}\}\ \ \ \ \text{ and }\ \ \ \ F_m = T^{r_m}\Big(J^{(k)}_j\Big).\]
Then $F_{m-1}$ and $F_m$ are both subsets of $J^{(m-1)}$, and $F_{m-1} = T^{r_{m-1}-r_m}(F_m)$, so by definition 
of the induced map $T_{m-1}$, $F_{m-1} = T_{m-1}^{l_m}(F_m)$ for some $l_m \geq 0$.

Let $j_m$ be such that $F_m \subset J^{(m)}_{j_m}$, let $j_k = j$. By maximality of $r_m$, 
\[ \forall\ 0 < l \leq l_m,\ T_{m-1}^l (F_m) \not\subset J^{(m)},\] 
hence $l_m < q^{(m)}_{j_m}$. Thus $e_m = (j_m,l_m)$ is an edge in $\mathcal{E}_m$.

Finally, the sequence $p_k = (e_1,e_2,\dots,e_k)$ is admissible. Indeed, recall that $s(e_m) = i_m$ where $i_m$ is such that 
$T_{m-1}^{l_m}\left( J^{(m)}_{j_m}\right) \subset J^{(m-1)}_{i_m}$. But then from the definitions of $j_m$ we see that $i_m = j_{m-1}$.

Thus $F = J(p_k)$, and the construction is unique such that $l_m < q^{(m)}_{j_m}$ for all $m$, hence the unique which gives a path in $\Sigma_k$.
\end{proof}

\begin{lemma}
Let $p\in \Sigma$, and for $k\geq 1$, let $p_k\in \Sigma_k$ be its truncation to the first $k$ edges. 
Then for all $k$, $J(p_{k+1}) \subset J(p_k)$. Hence the map $\pi:\Sigma \to \K$,
\[\pi(p) = \bigcap_{k\geq 1} J(p_k)\]
is well-defined.
\end{lemma}
\begin{proof}
By admissibility of $p$, \[T_k^{(l_{k+1})} \big(J^{(k+1)}_{j_{k+1}}\Big) \subset J^{(k)}_{j_{k}}.\]
Hence 
\[J(p_{k+1}) =  T^{l_1} \circ T_1^{l_2} \circ \dots \circ T_{k-1}^{l_k} \left( T_{k}^{l_{k+1}} \Big(J^{(k+1)}_{j_{k+1}}\Big) \right)  \subset
T^{l_1} \circ T_1^{l_2} \circ \dots \circ T_{k-1}^{l_k} \Big(J^{(k)}_{j_{k}}\Big) = J(p_k).\]

As shown in Lemma \ref{paths are floors}, $J(p_k)$ is a floor of the tower $Z^{(0,n_k)}_{j_k}$, which is a clopen subset of $K$, hence 
$\pi(p) = \bigcap_{k\geq 1} J(p_k)$ is non-empty. Further,
\[m_K(J(p_k)) = m_K\big(J^{(k)}_{j_k}\big) < m_K(J^{(k)}).\]
Since $m_K(J^{(k)}) \to 0$ as $k\to +\infty$, we see that the intersection $\pi(p)$ is indeed a single point in $\K$.
\end{proof}

The map $\pi$ is bijective, since for $x\in K$, $\pi^{-1}(x)$ can be uniquely determined by looking at the floors in the towers $Z_{j_k}^{(0,n_k)}$ that 
$x$ belongs to.

\subsubsection{Adic maps for IETs}
Once we have an ordered Bratteli diagram like above, we can define the adic map $\tau: \Sigma\setminus\Sigma_{max} \to \Sigma$. 

Define $\Sigma^0 = \Sigma\setminus \cup_{n \geq 0} \tau^{-n}(\Sigma_{max}) \setminus \cup_{n \geq 0} \tau^{n}(\Sigma_{min}) $. 
Then $\tau$ can be iterated arbitrarily many times both forwards and backwards on $\Sigma^0$.

\begin{lemma}\label{finite coding} 
For any $p\in \Sigma^0$,
\[\pi(\tau p) = T(\pi(p)).\]

\end{lemma}
\begin{proof}
Let $p = (e_1,e_2,\dots) \in \Sigma^0$, with $e_k=(j_k,l_k)$. Let $n = \min\{k: e_k \text{ is not maximal}\}$.
Then 
\[\tau p = (f_1,\dots,f_n,e_{n+1},e_{n+2},\dots),\ \text{ where  }\ f_n = (j_n,l_n+1),\ \text{ and for }\ k<n,\ f_k = (s(f_{k+1}),0).\]

Now, since $e_k$ is maximal for $k < n$, $l_k = q^{(k)}_{j_k} - 1$. Hence $T_{k-1} \circ T_{k-1}^{l_k} = T_{k-1}^{q^{(k)}_{j_k}} = T_k$. Thus:

\begin{align*}
T(J(p_n)) & = T \circ T^{l_1} \circ T_1^{l_2} \circ \dots \circ T_{n-1}^{l_n} \Big( J^{(n)}_{j_n}\Big) \\
&= T_1 \circ T_1^{l_2} \circ \dots \circ T_{n-1}^{l_n} \Big( J^{(n)}_{j_n}\Big) \\
&= T_2 \circ T_2^{l_3} \circ \dots \circ T_{n-1}^{l_n} \Big( J^{(n)}_{j_n}\Big) \\
&\dots\\
&= T_{n-1}^{l_n+1} \Big( J^{(n)}_{j_n}\Big) \\
& = J(\tau p_n).
\end{align*}

Similarly for any $k>n$, 
\[T(J(p_k)) = J(\tau p_k),\]
and so $T(\pi(p)) = \pi(\tau p)$.
\end{proof}

Hence we can use $\tau:\Sigma^0 \to \Sigma^0$ to code $T$, which will be the main tool in our argument.\\

It will be useful later to understand what paths we are removing by restricting our attention to $\Sigma^0$ instead of $\Sigma$.

\begin{lemma}\label{lem:minmax}
The set $\pi(\Sigma\setminus \Sigma^0)$ is {equal to} the $T$-orbit of $0\in \K$. In particular, $\Sigma\setminus \Sigma^0$ is countable.
\end{lemma}
\begin{proof}
Consider a maximal path $p\in \Sigma_\textrm{max}$. Then any truncation $p_k$ is maximal, which 
means that $\pi(J(p_k))$ is the top floor of its tower $Z_j^{(0,n_k)}$, meaning that $T(J(p_k)) \subset J^{(k)}$. As this holds for any $k$, $T(\pi(p))$ must
be contained inside $\cap_{k\geq 1} J^{(k)} = \{0\}$. Hence $\Sigma_\textrm{max} \subset \pi^{-1}(T^{-1}(0))$.
 
Similarly, consider a minimal path $p \in \Sigma_\textrm{min}$. For any $k$, the truncation $p_k$ is minimal, implying that $J(p_k)$ is contained in the 
base of its tower, hence in $J^{(k)}$. Then $\pi(p) \in \{0\}$, so $\Sigma_\textrm{min} \subset \pi^{-1} (0)$. 

{Since $\pi^{-1}(0)$ is minimal and $\pi^{-1}(T^{-1}(0))$ is maximal, $\Sigma\setminus \Sigma^0 = \pi^{-1}\big(\cup_{n\in \Z} T^n(0) \big)$. }
\end{proof}

\section{Symbolic renormalisation}\label{sec:symb}
The interpretation of the shift operator on a Bratteli diagram as renormalisation was made in Lindsey's thesis \cite{Lindsey_thesis} and in 
\cite{LT} in the context of Teichm\"uller geodesic flow. Here we explain this connection in the context of Rauzy-Veech renormalisation for IETs, 
and further extend it to skew-products.
For simplicity of notation, we deal with the case of periodic type skew-products, but the non-periodic case is analogous, as 
summarised in section \ref{subsec:non-periodic}.

From now on, fix an infinitely complete IET $\F\in \mathcal{X}_d^0$ and a skewing cocycle $\phi: I \to G \cong \Z^m$, such that for some 
$N>0$, $\widetilde{\mathcal{R}}^N (\F_\phi) = \F_\phi$. {Let $T_\phi$ be the continuous extension on $K\times G$.}
Construct a Bratteli diagram with \[J^{(k)} = \K^{(kN)},\] let $\tau: \Sigma^0\to\Sigma^0$ be 
the corresponding adic map. \\

Since $\widetilde{\mathcal{R}}^N(\F_\phi) = \F_\phi$, also $\mathcal{R}^N(\F) = \F$. This means that if 
\[{s:I^{(N)}} \to I\] is the orientation-preserving 
dilation map that scales {$I^{(N)}$} up to $I$, then
\[\F_1 = \hat{\mathcal{R}}^N(\F) = \s^{-1} \circ \F \circ \s.\]

The by induction, for all $k$
\begin{equation*}\label{scaling T}
\F_k = \hat{\mathcal{R}}^{kN}(\F) = \s^{-k} \circ \F \circ \s^k.
\end{equation*}

Then we see that for all $k,j$, 
\begin{equation*}\label{scaling J}
s\big(I^{(N(k+1))}\big) = I^{(Nk)},\ \ \text{ and }\ \ s \big(I^{(N(k+1))}_j\big) = I^{(Nk)}_j.
\end{equation*} 

Define also the map $\mathcal{S}: \1 J \to K$ performing the role of normalisation for the continuous extension, given by
\[\mathcal{S} = l \circ s \circ l^{-1}.\]
Then for all $k$ and $j$,
\[T_k = \mathcal{S}^{-k} \circ T \circ \mathcal{S}^k,\]
as well as
\[\mathcal{S}(J^{(k+1)}) = J^{(k)},\ \ \text{ and }\ \ \mathcal{S} (J^{(k+1)}_j) = J^{(k)}_j.\]

Define $q_j = q^{(1)}_j = q^{(0,N)}_j$ for all $j$, and note that $q^{(k)}_j = q_j$ for all $k$.

Let $A = A^{(0,N)}$ be the matrix of the Rauzy-Veech cocycle for $T$. We assume that $A$ is positive (or else we take a multiple of $N$).

Since $\widetilde{\mathcal{R}}^N(\F_\phi) = \F_\phi$, we have from Proposition \ref{height cocycle}
\[{\phi^{(N)} =\ } A^T \phi = \phi.\]

\subsection{The shift map}
\begin{defn}
Let $\sigma : \Sigma \to \Sigma$ be the left shift map
\begin{align*}
\sigma : \Sigma &\to \Sigma\\
(e_1,e_2,e_3,\dots) &\mapsto (e_2,e_3,\dots).
\end{align*}

Let $\iota: \Sigma \to \Sigma$ be the right shift map
\begin{align*}
\iota : \Sigma &\to \Sigma\\
(e_2,e_3,\dots) &\mapsto \big((s(e_2),0),e_2,e_3,\dots\big).
\end{align*}

\end{defn}
The dynamical interpretation of $\iota$ is as the {inverse of the normalisation map, namely the map
$\mathcal{S}^{-1}: K \to \1J \subset K$.}
On the other hand $\sigma$ corresponds to projection of the Rokhlin towers $\1Z_j$ onto their bases, 
followed by {normalisation} back up to $\K$ using $\mathcal{S}$. Formally:

\begin{lemma}\label{lem:projection}
\begin{enumerate}[a)]
\item For all $p\in \Sigma$,
\[\pi(\iota p) = \mathcal{S}^{-1}(\pi(p)).\]

\item Consider the map $R: \K \to \1J$ given by $R(x) = T^{-l}(x)$, where
\[l = \min_{n \geq 0} \{n: T^{-n}(x) \in \1J\}.\]
Then for all $p\in \Sigma$,
\[\pi(\sigma p) = \mathcal{S}\circ R\,(\pi(p)).\]
\end{enumerate}
\end{lemma}
\begin{proof}
\begin{enumerate}[a)]
\item Consider any truncation $p_r$ of $p$. From the definition of $\pi$ and the equality $T_{k+1} = \mathcal{S}^{-1} \circ T_k \circ \mathcal{S}$ 
for all $k$, it follows that $J(\iota p_r) = \mathcal{S}^{-1} J(p_r).$ Taking the intersection as $r\to \infty$ we conclude that 
$\pi(\iota p) = \mathcal{S}^{-1}(\pi(p)).$
\item Let $p = ((j_1,l_1),e_2,e_3,\dots)$. Then $\iota \sigma p = ((j_1,0),e_2,e_3,\dots)$. From the definition of $\pi$,
\[\pi(\iota \sigma p ) = T^{-l_1} (\pi(p)) = R(\pi(p)).\]
We conclude using part (a).
\end{enumerate}
\end{proof}

\subsection{Symbolic renormalisation for skew-products}
We code the skew-product $T_\phi$ by the skew-product
\begin{align*}
\tau_\phi: \Sigma^0 \times G &\to \Sigma^0 \times G\\
(p,a) &\mapsto (\tau p,a+\phi(p)).
\end{align*}
Here to simplify notation we denote by $\phi: \Sigma \to G$ the composition of $\pi$ with $\phi: \K \to G$. Note that if $p = (e_1,e_2,\dots)$, then 
$\pi(p) \in \K_{s(e_1)}$, so $\phi(p) = \phi_{s(e_1)}$.\\

We will now define a cocycle $f$ for the shift $\sigma$ such the skew-product $\sigma_f$ will correspond to the projection of the 
towers $\1{\tilde Z}_{j,a}$ onto their bases, in the same way as $\sigma$ corresponds to the projection of the towers $\1Z_j$.

\begin{defn}\label{defn:f}
Define the function $f:\Sigma \to G$ as follows:

If $p = ((j_1,l_1),\dots)$, then 
\[f(p) = -\sum_{l=0}^{l_1-1} \phi_{s((j_1,l))} = -S^T_{l_1} \phi \big|_{\1J_{j_1}}.\]

\end{defn}
See the figure below for a pictorial representation of this definition.
\begin{figure}[H]
\centering
\begin{tikzpicture}[scale=5]
	\draw[line width=1] (0,0) -- (1,0);
	\draw[line width=0.5] (0,0.1) -- (1,0.1);
	\draw[line width=0.5] (0,0.2) -- (1,0.2);
	\draw[line width=0.5] (0,0.3) -- (1,0.3);
	\draw[line width=0.5] (0,0.4) -- (1,0.4);
	\draw[line width=0.5] (0,0.5) -- (1,0.5);
	
	\node[] at (1.15,0) {$\subset \K_{i_0}$};
	\node[] at (1.15,0.1) {$\subset \K_{i_1}$};
	\node[] at (1.15,0.2) {$\subset \K_{i_2}$};
	\node[] at (1.15,0.3) {$\subset \K_{i_3}$};
	\node[] at (1.15,0.4) {$\subset \K_{i_4}$};
	
	\node[] at (1.45,0) {$f = 0$};
	\node[] at (1.51,0.1) {$f = -\phi_{i_0}$};
	\node[] at (1.605,0.2) {$f = -\phi_{i_0} - \phi_{i_1}$};
	\node[] at (1.70,0.3) {$f = -\phi_{i_0} - \phi_{i_1} - \phi_{i_2}$};
	\node[] at (1.795,0.4) {$f = -\phi_{i_0} - \phi_{i_1} - \phi_{i_2} - \phi_{i_3}$};
	\node[] at (1.89,0.5) {$f = -\phi_{i_0} - \phi_{i_1} - \phi_{i_2} - \phi_{i_3} - \phi_{i_4}$};

\end{tikzpicture}
\caption{The map $f$ on the floors of a tower $Z^{(1)}_j$.}
\label{f picture}
\end{figure}
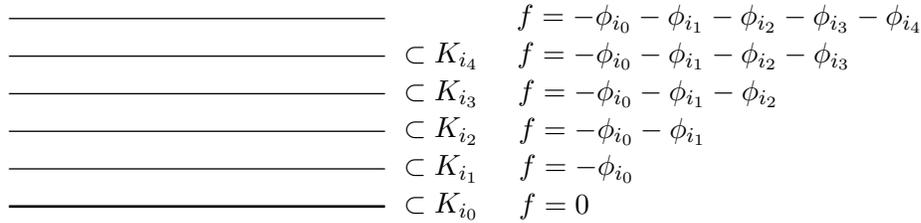

\begin{lemma}\label{lem:projection_sp}
For all $p\in \Sigma$,
Consider the map $\tilde R: \K\times G \to \1J\times G$ given by $\tilde R(x,a) = T_\phi^{-l}(x,a)$, where
\[l = \min_{n \geq 0} \{n: T^{-n}(x) \in \1J\}.\]
Then for all $(p,a)\in \Sigma \times G$,
\[(\pi\times \id) (\sigma_f (p,a)) = (\mathcal{S}\times \id)\circ \tilde R\,(\pi(p),a).\]
\end{lemma}
(Recall that $\mathcal{S}:\1J \to \K$ is the {normalising} map.)
\begin{proof}
Observe that $T_\phi^{-1}(x,a) = \big(T^{-1}(x),a-\phi(T^{-1}(x))\big).$
Hence if $x = \pi(p)$, then
\[\widetilde{R}(x,a) = T_\phi^{-l}(x,a) = \Big(T^{-l},a-\sum_{i=1}^{l} \phi\Big(T^{-i}(x)\Big)\Big) = \big(T^{-l}(x),a+f(p)\big) = (R(x),a+f(p)).\]
We conclude from Lemma \ref{lem:projection}.
\end{proof}

Recall the definitions of the orbit equivalence relation $\mathcal{O}$ and tail equivalence relation $\mathfrak{T}$ from section \ref{subsec:ideaANSS}.
\begin{prop}\label{prop:tail=orb}
The following equality of relations holds:
\begin{equation*}
\mathcal{O}(\tau_{\phi}) = \mathfrak{T}(\sigma_f).
\end{equation*}
\end{prop}
\begin{proof}
Applying Lemma \ref{lem:projection_sp} $k$ times, we see that $\sigma_f^k$ codes the projection of the towers $\tilde{Z}^{(0,kN)}_{j,a}$ onto their 
bases:\\
If $\tilde R^{(k-1)}: \K\times G \to J^{(k)}\times G$ is given by $\tilde R^{(k-1)}(x,a) = T_\phi^{-l_k}(x,a)$, where
\[l_k = \min_{n \geq 0} \{n: T^{-n}(x) \in J^{(k)}\},\]
then for all $(p,a)\in \Sigma \times G$,
\[(\pi\times \id) (\sigma_f^k(p,a)) = (\mathcal{S}\times \id)^k\circ \tilde R^{(k-1)}\,(\pi(p),a).\]

\bigskip
First let us show the inclusion $\mathcal{O}(\tau_{\phi}) \subset \mathfrak{T}(\sigma_f)$.\\
Suppose $p,p' \in \Sigma^{0},\ a,a' \in G$ are such that for some $n>0$, $\tau_\phi^n(p,a) = (p',a')$. If $x=\pi(p), x'=\pi(p')$, 
there exists some sufficiently high ${k}>0$ such 
\[\{T_\phi^r (x,a) : r=1,\dots,n\} \cap J^{({k})}\times G = \emptyset.\]
Then $(x,a)$ and $(x',a')$ are directly above each other in some tower $\tilde{Z}^{(0,{k}N)}_{j,b}$, and hence project down onto the same point. 
Hence $\sigma_f^{{k}}(p,a) = \sigma_f^{{k}}(p',a')$.

\bigskip
Now let us show the converse inclusion $\mathcal{O}(\tau_{\phi}) \supset \mathfrak{T}(\sigma_f)$.\\
Suppose $p,p' \in \Sigma^{0},\ a,a' \in G$ are such that for some $k>0$, $\sigma_f^k(p,a) = \sigma_f^k(p',a')$. Let $x = \pi(p), x'=\pi(p')$. 
Then $(x,a)$ and $(x',a')$ project down to the same point in the base of a tower $\tilde{Z}^{(0,kN)}_{j,b}$, so they must be in the same $T_\phi$ 
orbit. Hence $(p,a)$ and $(p',a')$ are in the same $\tau_\phi$ orbit.

\end{proof}

\begin{minipage}{0.37\textwidth}
\begin{figure}[H]
\centering
\begin{tikzpicture}[scale=5, vertex/.style={circle,inner sep=2pt,fill=blue}]
	\draw[line width=1] (0,0) -- (1,0);
	\draw[line width=0.5] (0,0.1) -- (1,0.1);
	\draw[line width=0.5] (0,0.2) -- (1,0.2);
	\draw[line width=0.5] (0,0.3) -- (1,0.3);
	\draw[line width=0.5] (0,0.4) -- (1,0.4);
	\draw[line width=0.5] (0,0.5) -- (1,0.5);
	\draw[line width=0.5] (0,0.6) -- (1,0.6);
	
	\node[vertex] at (0.3,0.2) {};
	\node[vertex] at (0.3,0) {};
	
	\node[blue] at (0.48,0.25) {$x=\pi(p)$};
	
	\node[blue] at (0.55,-0.07) {$R(x) = \pi(\iota\sigma p)$};
	
	\node[] at (-0.2,0.1) {$\1{Z}_{j}$};

\end{tikzpicture}
\captionsetup{width=.9\linewidth}
\caption{The map $R$ projects points in a tower $\1Z_j$ down onto the base of the tower. Symbolically it is given by $\iota\circ \sigma$.}
\label{fig:projection}
\end{figure}
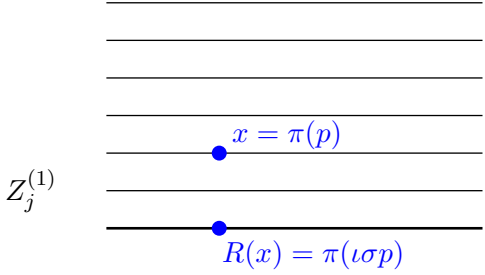
\end{minipage}\hspace{0.01\textwidth}
\begin{minipage}{0.64\textwidth}
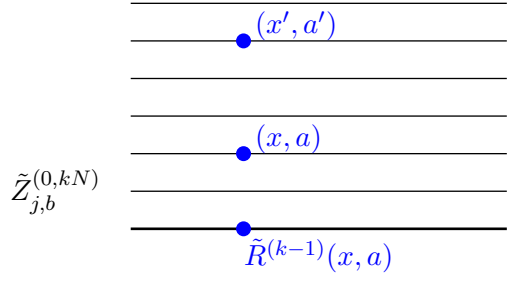
\begin{figure}[H]
\centering
\begin{tikzpicture}[scale=5, vertex/.style={circle,inner sep=2pt,fill=blue}]
	\draw[line width=1] (0,0) -- (1,0);
	\draw[line width=0.5] (0,0.1) -- (1,0.1);
	\draw[line width=0.5] (0,0.2) -- (1,0.2);
	\draw[line width=0.5] (0,0.3) -- (1,0.3);
	\draw[line width=0.5] (0,0.4) -- (1,0.4);
	\draw[line width=0.5] (0,0.5) -- (1,0.5);
	\draw[line width=0.5] (0,0.6) -- (1,0.6);
	
	\node[vertex] at (0.3,0.2) {};
	\node[vertex] at (0.3,0.5) {};
	\node[vertex] at (0.3,0) {};
	
	\node[blue] at (0.42,0.24) {$(x,a)$};
	\node[blue] at (0.44,0.54) {$(x',a')$};
	
	\node[blue] at (0.5,-0.07) {$\tilde{R}^{(k-1)}(x,a)$};
	
	\node[] at (-0.2,0.1) {$\tilde{Z}^{(0,kN)}_{j,b}$};

\end{tikzpicture}
\captionsetup{width=.85\linewidth}
\caption{Proof of Proposition \ref{prop:tail=orb}. 
Two points $(x,a)$ and $(x',a') \in \K \times G$ are in the same $T_\phi$ orbit if and only if for some $k>0$ they are in the same 
tower $\tilde{Z}^{(0,kN)}_{j,b}$ and project down onto the same point in the base.}
\label{fig:projection}
\end{figure}
\end{minipage}

\subsection{The non-periodic case}\label{subsec:non-periodic}
In the non-periodic case, Proposition \ref{prop:tail=orb} still holds. In this section we briefly give the necessary notation and statements of 
the analogues of Lemmas \ref{lem:projection} and \ref{lem:projection_sp} needed for the proof. The proofs follow the proofs in the periodic case.

If the IET $T$ is not periodic under renormalisation, the corresponding 
Bratteli diagram is non-stationary, so it is important to keep track of the level of the diagram in which one is working, this is the main difference 
in the notation.\\

Choosing some increasing sequence $(n_k)$, construct a Bratteli diagram for $T$ with $J^{(k)} = \K^{(n_k)}$.
For $k\geq 0$, denote by $\nk\Sigma$ the set of admissible paths in $\prod_{r\geq k+1} \mathcal{E}_r$, i.e. paths starting at level $k$. 
Correspondingly define the map $\nk\pi: \nk\Sigma \to \nk J$. The adic map $\nk\tau$ on $\nk\Sigma$ then codes $T_k$, in the sense that
$\nk\pi \circ \nk\tau = T_k \circ \nk\pi$.

Introduce also the following definitions, analogously to the periodic case:
\begin{itemize}
	\item The left shift $\sigma: \nk\Sigma \to \Sigma^{(k+1)}$ and the right shift $\iota:\Sigma^{(k+1)} \to \nk\Sigma$.
	\item The map $\nk R: \K \to J^{(k+1)}$, representing the projection to the base of the tower $Z_j^{(0,n_{k+1})}$, 
	is defined as $\nk R(x) = T^{-l}(x)$ for $l = \min\{m\geq 0: T^{-m}(x) \in J^{(k+1)}\}$.
	\item Similarly the projection in the tower for the skew-product, 
	$\nk{\widetilde{R}}: \K \times G \to J^{(k+1)}\times G$ is defined as $\nk{\widetilde R}(x,a) = T_\phi^{-l}(x,a)$.
	\item The function $\nk f: \nk \Sigma \to G$, defined as $\nk f(p) = -S^{T_k}_{l_{k+1}} \phi^{{(n_k)}} \big|_{\nk J_{j_{k+1}}}$ 
	for $p = \big((j_{k+1},l_{k+1}),\dots \big) \in \nk \Sigma$.
	\item Denote by $\sigma_f$ the map that for any $k$ sends $(p,a) \in \nk \Sigma \times G$ to $(\sigma p,a+\nk f(p))$.

\end{itemize}

Then the analogues of Lemmas \ref{lem:projection} and \ref{lem:projection_sp} are:
\begin{lemma}
\begin{enumerate}[a)]
	\item For $p\in \Sigma^{(k+1)}$,
	\[\nk\pi (\iota p) = \pi^{(k+1)}(p).\]
	\item For $p\in \Sigma = \Sigma^{(0)}$ and $k\geq 1$, 
	\[\nk\pi (\sigma^k p) = R^{(k-1)} (\pi (p)).\]
	\item For $(p,a) \in \Sigma \times G$ and $k\geq 1$, 
	\[(\nk\pi \times \id)(\sigma^k_f(p,a)) = {\widetilde R}^{(k-1)}(\pi(p),a).\]
\end{enumerate}
\end{lemma}

\section{Applying the ANSS theorem}\label{sec:ANSS}
Here we reproduce the result from \cite{ANSS} and use it to deduce Theorem \ref{thm:main}.
\subsection{Statement of the theorem}

\begin{defn}[\cite{ANSS}]\label{defn:phi_f}
Given a function $f: \Sigma \to G$, define for $p\in \Sigma\setminus\Sigma_\textrm{max}$ 
\[\phi_f(p) = \sum_{i=0}^{\infty} \big(f(\sigma^i p ) - f(\sigma^i \tau p )\big).\] 
The map $\phi_f$ is well-defined on $\Sigma\setminus\Sigma_\textrm{max}$:
the sum is finite since by definition of the adic map $\tau$, 
$p$ and $\tau p$ lie in the same tail equivalence class of $\sigma$, hence for some $N$, $\forall n \geq N,\ \sigma^n(p) = \sigma^n(\tau p)$.
\end{defn}

By construction of $\phi_f$, the following equality of relations is satisfied:
\begin{equation}\label{eqn:tail=orb}
\mathcal{O}(\tau_{\phi_f}) = \mathfrak{T}(\sigma_f).
\end{equation}
For a proof of this equality see Appendix \ref{app:orb=tail}.\\

Note that from Proposition \ref{prop:tail=orb} it immediately follows that 
\begin{equation}\label{eqn:orb=orb}
\mathcal{O}(\tau_{\phi_f}) = \mathcal{O}(\tau_{\phi}),
\end{equation}
allowing us to deduce Theorem \ref{thm:main} from the theorem of \cite{ANSS} (Theorem \ref{ANSS thm} below).

In fact it is true 
that $\phi_f = \phi$, so $\tau_{\phi_f}$ and $\tau_\phi$ are the same map, but this is not needed for our proof of Theorem \ref{thm:main}. 
Nevertheless we give a proof of 
this equality in Appendix \ref{app:phi=phi_f} for completeness.

\begin{defn}
Say a function $f:\Sigma \to G$ is \textbf{finite memory} if there exists a $K$ such that for all $p = (e_1,e_2,\dots) \in \Sigma$,
$f(p) = f(e_1,\dots,e_K)$.\\

Say $f$ is \textbf{periodic} if there exist:
\begin{itemize}
\item A non-constant measurable $g:\Sigma \to S^1$,
\item A homomorphism $\gamma \in \hat G = \textrm{Hom}(G,S^1)$,
\item and a $z\in S^1$, such that
\[\gamma \circ f = z \frac{g\circ \sigma}{g}.\]
If $f$ is not periodic, say that it is \textbf{aperiodic}.

\end{itemize}
\end{defn}

\begin{thm}\cite[Theorems 2.1 and 2.2]{ANSS}\label{ANSS thm}
Suppose that $\sigma:\Sigma\to\Sigma$ is topologically mixing, $f:\Sigma^0 \to G$ is finite memory and aperiodic. Then:
\begin{enumerate}[(I)]
		\item For every homomorphism $\psi: G \to \R$ there exists a unique (up to scaling) measure $\nu_\psi$ on $\Sigma^0$ which is quasi-invariant under $\tau$ with 
		$\dd{\nu_\psi \circ \tau}{\nu_\psi} = e^{-\psi \circ \phi_f}$, and this measure is non-atomic. 
		
		Given such a $\nu_\psi$ we define the Maharam measure $\mu_\psi$ on $\Sigma^0\times G$, with 
		\[\ud\mu_\psi(p,a) = \ud \nu_\psi(p) e^{\psi(a)} \ud \mu_G (a),\] where $\mu_G$ is the Haar measure on $G$.
		
		\item Each Maharam measure $\mu_\psi$ is an ergodic invariant Radon measure for $\tau_{\phi_f}$.
		\item Every ergodic invariant Radon measure for $\tau_{\phi_f}$ is up to scaling a Maharam measure for some homomorphism $\psi$.
	\end{enumerate}
\end{thm}

\begin{rem}
In fact $f$ having finite memory is required only for part (III) whereas for parts (I) and (II) (Theorem 2.1 in \cite{ANSS}) 
it suffices for $f$ to be H\"older.
\end{rem}

We will apply this theorem to the function $f$ defined in Definition \ref{defn:f}.
To be able to do this we need $f$ to be finite memory and aperiodic. $f$ is finite memory directly by construction, 
since it depends only on the first edge of a path in $\Sigma$. 
\subsection{Aperiodicity of $f$}

To check that the function $f$ is aperiodic, we use a similar strategy to Pollicott and Sharp in \cite[Proposition 3.1]{PS}.\\

For $n \geq 1$, let $\textrm{Fix}_n(\sigma) = \{p\in \Sigma^0: \sigma^n p = p\}$ be the set of fixed points of $\sigma^n$. Define the set
\[\Delta = \{S^\sigma_n f (p_1) - S^\sigma_n f (p_2): p_1,p_2 \in \textrm{Fix}_n(\sigma) \text{ for some } n\geq 1\}.\]
It is a standard fact that $\Delta$ is a subgroup of $G$. (For completeness we give a proof of this fact 
in Appendix \ref{app:Delta}.) Our aim will be to show that $\Delta = G$. 

Recall that we chose $N$ to be such that $\widetilde{\mathcal{R}}^N(\F_\phi) = \F_\phi$. 
However $N$ does not need to be the smallest period, we can take any multiple of the period as $N$ and this has no effect on all preceding arguments. 
The key here will be to take $N$ large enough to make the proof of $\Delta = G$ simple.

The larger $N$ is, the smaller is $\1J$. We will take $N$ sufficiently large that the towers $Z^{(1)}_j$ have the same beginning for all $j$, 
in the following sense:

\begin{lemma}\label{same tower lemma}
If $N$ is large enough, then there exists an $M>0$ for which the following holds:
\begin{itemize}
\item $\displaystyle \min_{j=1,...,d} q_j > M + 1$.
\item For $n \leq M$, $T^n\big(\1J_j\big) \subset \K_{i(n)}$ where $i(n)\in \{1,...,d\}$ only depends on $n$ and not on $j$.
\item The set $\{i(n): 1 \leq n \leq M-1\}$ contains each of the indices $1,...,d$.
\end{itemize}
\end{lemma}
\begin{proof}
For all $n\geq 0$, let $i(n)$ be such that 
\[ T^n(0) \in \K_{i(n)}.\]
Let $M$ be large enough that 
\[\{i(n): 1 \leq n \leq M-1\} = \{1,\dots,d\}.\]
Such an $M$ exists by infinite-completeness of $T$.

We take $N$ large enough such that $\1J$ is sufficiently small to ensure that 
\[\bigcup_{0 \leq n \leq M} T^n\big(\1J\big) \cap \{l(x_1),\dots,l(x_{d-1})\} = \emptyset.\]

This means that if $y$ is the right end-point of $\1J$, then for all $n\leq M$, there are no $l(x_i)$ between $T^n(0)$ and $T^n(y)$, and hence
\[ T^n\big(\1J\big) \subset \K_{i(n)}.\]

Finally we make sure $N$ is sufficiently large that $\min_{j=1,...,d} q_j > M + 1$.
\end{proof}

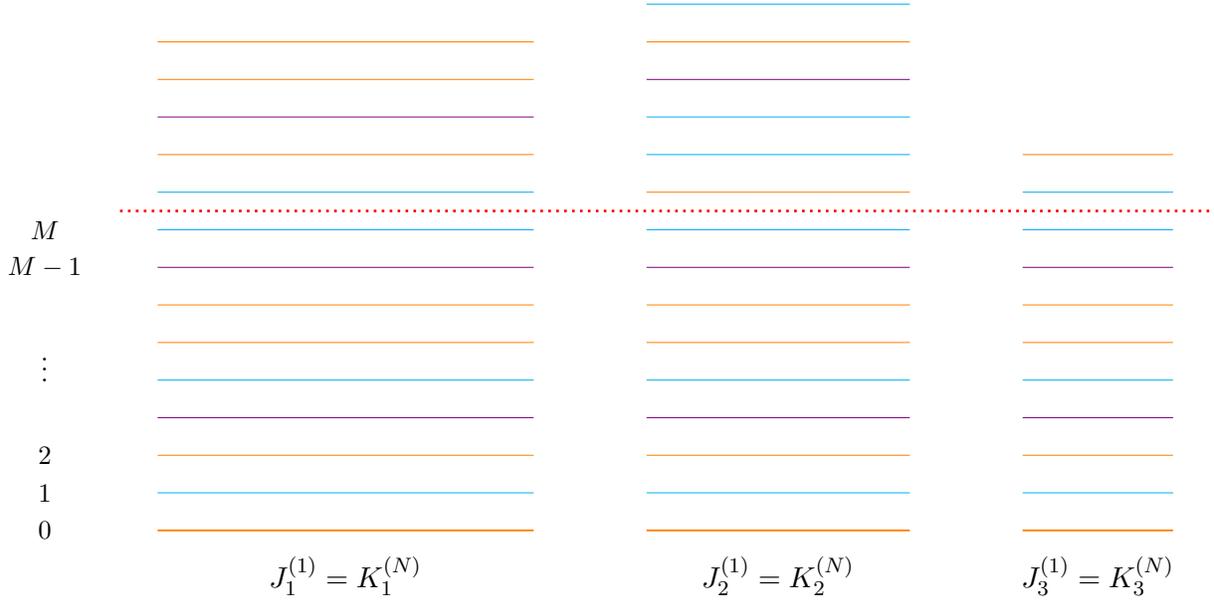
\begin{figure}[H]
\centering
\begin{tikzpicture}[scale=5]
	\draw[line width = 0.7,orange] (0,0) -- (1,0);
	\draw[line width = 0.3,cyan] (0,0.1) -- (1,0.1);
	\draw[line width = 0.3,orange] (0,0.2) -- (1,0.2);
	\draw[line width = 0.3,violet] (0,0.3) -- (1,0.3);
	\draw[line width = 0.3,cyan] (0,0.4) -- (1,0.4);
	\draw[line width = 0.3,orange] (0,0.5) -- (1,0.5);
	\draw[line width = 0.3,orange] (0,0.6) -- (1,0.6);
	\draw[line width = 0.3,violet] (0,0.7) -- (1,0.7);
	\draw[line width = 0.3,cyan] (0,0.8) -- (1,0.8);
	\draw[line width = 0.3,cyan] (0,0.9) -- (1,0.9);
	\draw[line width = 0.3,orange] (0,1) -- (1,1);
	\draw[line width = 0.3,violet] (0,1.1) -- (1,1.1);
	\draw[line width = 0.3,orange] (0,1.2) -- (1,1.2);
	\draw[line width = 0.3,orange] (0,1.3) -- (1,1.3);
	
	\draw[line width = 0.7,orange] (1.3,0) -- (2,0);
	\draw[line width = 0.3,cyan] (1.3,0.1) -- (2,0.1);
	\draw[line width = 0.3,orange] (1.3,0.2) -- (2,0.2);
	\draw[line width = 0.3,violet] (1.3,0.3) -- (2,0.3);
	\draw[line width = 0.3,cyan] (1.3,0.4) -- (2,0.4);
	\draw[line width = 0.3,orange] (1.3,0.5) -- (2,0.5);
	\draw[line width = 0.3,orange] (1.3,0.6) -- (2,0.6);
	\draw[line width = 0.3,violet] (1.3,0.7) -- (2,0.7);
	\draw[line width = 0.3,cyan] (1.3,0.8) -- (2,0.8);
	\draw[line width = 0.3,orange] (1.3,0.9) -- (2,0.9);
	\draw[line width = 0.3,cyan] (1.3,1) -- (2,1);
	\draw[line width = 0.3,cyan] (1.3,1.1) -- (2,1.1);
	\draw[line width = 0.3,violet] (1.3,1.2) -- (2,1.2);
	\draw[line width = 0.3,orange] (1.3,1.3) -- (2,1.3);
	\draw[line width = 0.3,cyan] (1.3,1.4) -- (2,1.4);
	
	\draw[line width = 0.7,orange] (2.3,0) -- (2.7,0);
	\draw[line width = 0.3,cyan] (2.3,0.1) -- (2.7,0.1);
	\draw[line width = 0.3,orange] (2.3,0.2) -- (2.7,0.2);
	\draw[line width = 0.3,violet] (2.3,0.3) -- (2.7,0.3);
	\draw[line width = 0.3,cyan] (2.3,0.4) -- (2.7,0.4);
	\draw[line width = 0.3,orange] (2.3,0.5) -- (2.7,0.5);
	\draw[line width = 0.3,orange] (2.3,0.6) -- (2.7,0.6);
	\draw[line width = 0.3,violet] (2.3,0.7) -- (2.7,0.7);
	\draw[line width = 0.3,cyan] (2.3,0.8) -- (2.7,0.8);
	\draw[line width = 0.3,cyan] (2.3,0.9) -- (2.7,0.9);
	\draw[line width = 0.3,orange] (2.3,1) -- (2.7,1);
	
	\draw[dotted,red,line width = 1] (-0.1,0.85) -- (2.8,0.85);
	
	\node at (-0.3,0) {\small $0$};
	\node at (-0.3,0.1) {\small $1$};
	\node at (-0.3,0.2) {\small $2$};
	
	\node at (-0.3,0.45) {\small $\vdots$};
	
	\node at (-0.3,0.7) {\small $M-1$};
	\node at (-0.3,0.8) {\small $M$};
	
	\node at (0.5,-0.12) {$\1J_1 = \K^{(N)}_1$};
	\node at (1.65,-0.12) {$\1J_2 = \K^{(N)}_2$};
	\node at (2.5,-0.12) {$\1J_3 = \K^{(N)}_3$};

\end{tikzpicture}
\caption{Towers $\1Z_j$ with the same beginning - the different sets $\K_i$ containing the floors are shown in different colours.}
\end{figure}

\begin{prop}
If $N$ is chosen sufficiently large that the conditions in Lemma \ref{same tower lemma} hold, then $\Delta = G$.
\end{prop}
\begin{proof}
For $1 \leq n\leq M$, 
\[T^n \Big(\1J_{i(n)}\Big) \subset \K_{i(n)},\] so
the edge $e = (i(n),n)\in \mathcal{E}$ has $s(e)=t(e)=i(n)$.\\

Hence the path $p(n) = (e,e,e,...)$ is admissible and a fixed point of $\sigma$, and as $0<n<q_{i(n)}-1$, the edge $e$ is not maximal, 
so the path $p(n)$ is not maximal.
Define such a path $p(n)$ for all $1\leq n \leq M$.\\

Then 
\[f(p_n) = -S^T_n \phi\ \big|_{J^{(1)}_{i(n)}} = -\sum_{r=0}^{n-1} \phi_{i(r)}.\]
As $p(n),p(n+1) \in \textrm{Fix}_1(\sigma)$, 
\[f(p(n))-f(p(n+1)) = \phi_{i(n)} \in \Delta.\]

Thus $\Delta$ contains $\{\phi_{i(n)}: 1\leq n\leq M-1\} = \{\phi_j: j=1,...,d\}$. As we assumed that the set of values
$\{\phi_j: j=1,...,d\}$ generates $G$, we deduce that $\Delta = G$.
\end{proof}

Finally we deduce that $f$ is aperiodic.
 
\begin{prop}\label{prop:aperiodic}
If $N$ is chosen sufficiently large, then the function $f$ is aperiodic.
\end{prop} 
\begin{proof}
Suppose for contradiction that $f$ is periodic, meaning that there exist:

\begin{itemize}
\item A non-constant measurable $g:\Sigma \to S^1$,
\item A homomorphism $\gamma \in \hat G = \textrm{Hom}(G,S^1) \cong \mathbb{T}^m$,
\item and a $z\in S^1$, such that
\begin{equation}\label{periodicity equation}
\gamma \circ f = z \frac{g\circ \sigma}{g}.
\end{equation}
\end{itemize}
\bigskip

The homomorphism $\gamma$ must have the form
\[\gamma(a) = e^{2\pi i \langle a, b \rangle},\ \text{ for some } b \in \R^m.\]

Suppose $p \in \mathrm{Fix_n}(\sigma)$. Then $g(\sigma^n p) = g(p)$, so 
iterating the recurrence (\ref{periodicity equation}) $n$ times, we deduce

\begin{equation}\label{z equation}
 z^n = z^n \frac{g\big(\sigma^{n}(p)\big)}{g(p)} = \prod_{i=0}^{n-1} z\frac{g\big(\sigma^{i+1}(p)\big)}{g\big(\sigma^i(p)\big)} = \prod_{i=0}^{n-1} \gamma \circ f\big(\sigma^i(p)\big)
 = \scalebox{1.2}{ $e^{2\pi i \langle S^\sigma_n f(p), b\rangle}$ }.
\end{equation}
 
For any $\delta \in \Delta = G$, we have $\delta = S^\sigma_n f\ (p_1) - S^\sigma_n f\ (p_2)$ for some $n$ and 
$p_1,p_2 \in \text{Fix}_n(\sigma)$, and so
\[\scalebox{1.2}{$\displaystyle e^{2\pi i\langle \delta, b\rangle} = \frac{e^{2\pi i \langle S^\sigma_n f(p_1), b\rangle}}{e^{2\pi i \langle S^\sigma_n f(p_2), b\rangle}} 
= \frac{z^n}{z^n} = 1$}.\]
 
Thus
\[ \langle \delta, b\rangle \in \Z \ \text{ for all } \ \delta  \in \Delta = G = \Z^m.\]

Hence we deduce that $b\in \Z^m$, meaning that $\gamma$ is trivial. 
As $\rm{Fix}_1(\sigma) \neq \emptyset$, the equation (\ref{z equation}) implies that $z=1$ and therefore $g$ must be constant.

Thus we have a contradiction and so $f$ is aperiodic.
\end{proof}

\subsection{Proof of Theorem \ref{thm:main}}
\begin{proof}

Proposition \ref{prop:aperiodic} allows us to apply Theorem \ref{ANSS thm} with our function $f$ to 
deduce that the ergodic invariant Radon measures for $\tau_{\phi_f}$ are precisely the Maharam measures, which are non-atomic.

We know that the orbit equivalence relations of $\tau_\phi$ and $\tau_{\phi_f}$ are equal (\ref{eqn:orb=orb}),
 so the invariant sets of $\tau_{\phi_f}$ and $\tau_\phi$ are 
the same, and hence the ergodic measures are the same. Hence we deduce that the ergodic invariant Radon measures for 
$\tau_{\phi}\big|_{\Sigma^0\times G}$ are precisely 
the Maharam measures.

What is left to show is that there is a bijection between the invariant Radon measures for $\tau_\phi\big|_{\Sigma^0\times G}$ and those for 
$T_\phi\big|_{K\times G}$. \\

If $K^0 = \pi(\Sigma^0)$, then $\pi\times \id$ is a conjugacy between $\tau_\phi\big|_{\Sigma^0\times G}$ and $T_\phi\big|_{K^0\times G}$. As $\pi$ is a 
homeomorphism, pushing forward or pulling back a measure by $\pi$ preserves the property of being Radon, so there is a bijection between the 
invariant Radon measures of $\tau_\phi\big|_{\Sigma^0\times G}$ and $T_\phi\big|_{K^0\times G}$.
 Thus we conclude that the invariant Radon
measures of $T_\phi\big|_{K^0\times G}$ are precisely the Maharam measures.

What remains is to show that there is a bijection between the invariant Radon measures of 
$T_\phi\big|_{K^0\times G}$ and $T_\phi\big|_{K\times G}$. 
By Lemma \ref{lem:minmax}, the set $\K \times G \setminus (\K^0\times G)$ is countable, so if we show that there are no atomic invariant Radon 
measures supported on $\K \times G \setminus (\K^0\times G)$, then such a bijection is simply given by restriction and extension of the 
measures. By  Lemma \ref{lem:minmax}, $\K \times G \setminus (\K^0\times G)$ is in fact contained in the $T_\phi$ orbit of the set $\{0\}\times G$. 
This is a countable union of orbits, hence it suffices to show that there is no invariant Radon measure supported on the orbit of $(0,a)$ 
for any $a\in G$. \\

{
Indeed, for any $k\geq 1$, by Lemma \ref{lem:BS} we see that
 $S^T_{q_1^{(0,kN)}} \phi (0) = \phi^{(kN)}(0) = \phi(0)$, where the second 
equality holds by the periodic type assumption $\phi^{(N)} = \phi$.} This means that for the increasing sequence of times $t_k = q_1^{(0,kN)}$, 
for all $k$ and any $a\in G$, $T_\phi^{t_k}(0,a)$ is contained in $\K\times \{a+\phi(0)\}$.

Hence for any $a\in G$, the orbit of $(0,a)$ returns infinitely many times to $\K \times \{a+\phi(0)\}$, so any invariant Radon
measure $\mu$ must have $\mu(\{(0,a)\}) = 0$. 

Thus we deduce that there is a bijection between the invariant Radon measures of $\tau_\phi\big|_{\Sigma^0\times G}$ and of $T_\phi\big|_{\K\times G}$. 
This bijection preserves ergodicity, so the ergodic invariant Radon measures for $T_\phi$ are precisely the Maharam measures. \qedhere

\end{proof}

\section{Weak-* continuity of Maharam measures depending on the parameter}\label{sec:weak*}
In this section, still in the setting of periodic type skew-products,
we give explicit formulas for the Maharam measures of sets which appear as floors in Rokhlin towers for $T_\phi$, and hence deduce that 
the Maharam measures $\mu_\psi$ depend weak-* continuously on the parameter $\psi$. \\

\subsection{Recurrence relation for measures of {floors}}

\begin{notation}
    Since $f(p)$ only depends on the first edge of a path $p\in \Sigma$, we will write $f(e)$ for $e\in \mathcal{E}$ to denote $f(p)$ where $p$ is any path that starts with $e$.
\end{notation}

Recall that the Rauzy-Veech cocycle $A^{(r,k)}_{ij}$ counts the number of floors of the tower $Z^{(r,k)}_{j}$ that belong to $\K_i$. We define an 
analogous cocycle for skew-products, where in addition we keep track of the $G$-coordinate.

\begin{defn}
Let
\[b^{(k)}_{ij,a} := \#\{\text{Floors F in } \tilde{Z}^{(0,kN)}_{j,a}: F \subset \K_{i,0}\}.\]

By Lemma \ref{paths are floors} for finite IETs, the floors of $Z^{(0,kN)}_j$  which are contained in $\K_i$ correspond to paths 
\[p_k\in \Sigma_k:\  s(p_k) = i,\ t(p_k) = j.\]

Then for the skew-product, by Lemma \ref{lem:projection_sp} applied $k$ times, the floors of $\tilde{Z}^{(0,kN)}_{j,a}$  which are contained in $\K_{i,0}$ correspond to paths
\[p_k\in \Sigma_k:\  s(p_k) = i,\ t(p_k) = j,\ S^\sigma_k f (p_k) = a.\]

Hence we can also interpret $b^{(k)}_{ij,a}$ as a count of paths in the Bratteli diagram:
\begin{equation}
b^{(k)}_{ij,a} = \#\left\{p_k\in \Sigma_k:\ s(p_k) = i,\ t(p_k) = j,\ S^\sigma_k f (p_k) = a  \right\}.
\end{equation}
\end{defn}
\begin{defn}
We can consider the numbers $b^{(k)}_{ij,a}$ as the coefficients of a generating function, by defining the $d\times d$ matrix $M^{(k)}$ which has as
entries Laurent polynomials in the variables $t = (t_1,\dots t_m)$:

\[\left(M^{(k)}\right)_{ij}(t) := \sum_{a\in G} b^{(k)}_{ij,a} t^a,\]

where for $a = (a_1,\dots,a_m)\in \Z^m$, we denote by $t^a$ the product $t_1^{a_1} t_2^{a_2} \cdot \ldots \cdot t_m^{a_m}$.\\

We call $M$ the \textbf{level-counting cocycle} to emphasise that in addition to counting the floors of towers belonging to the {sets} 
$\K_i$ as done by the Rauzy-Veech cocycle $A$, $M$ also records the \emph{level} where the floors live, where by level we mean the projection of 
$\K \times G$ onto $G$.
\end{defn}

\begin{rem}
Note that $M$ is related to $A$ in the following way. As the towers $\tilde{Z}^{(0,kN)}_{j,a}$ cover the towers $Z^{(0,kN)}_j$, 
\[\sum_{a\in G} b^{(k)}_{ij,a} = A^{(0,kN)}_{ij}.\]
So if $\mathbf{1}$ is the vector $(1,1,\dots,1)\in \Z^m$, then 
\[M^{(k)}(\mathbf{1}) = A^{(0,kN)}.\]
\end{rem}

\begin{lemma}
If $M = \1M$, then

\[M_{ij}(t) = \sum_{\substack{e \in \mathcal{E}\\ s(e) = i,\ t(e) = j}} t^{f(e)}.\]
\end{lemma}
\begin{proof}
This follows from the definition,
\[b^{(1)}_{ij,a} = \#\left\{p_1\in \Sigma_1 = \mathcal{E}:\ s(p_1) = i,\ t(p_1) = j,\ f (p_1) = a  \right\}.\]
\end{proof}

\begin{lemma}\label{M is cocycle}
We have
\[M^{(k)} = M^k.\]
\end{lemma}
\begin{proof}
We show this for $k=2$:
\begin{align*}
\left(M^2\right)_{ij} &= \sum_{r=1}^d M_{ir}M_{rj} \\
 &= \sum_{r=1}^d \left[ \left(\sum_{\substack{e \in \mathcal{E}\\ s(e) = i,\ t(e) = r}} t^{f(e)}\right)
\left(\sum_{\substack{e \in \mathcal{E}\\ s(e) = r,\ t(e) = j}} t^{f(e)}\right)  \right] \\
 &= \sum_{\substack{p = (e_1,e_2) \in \Sigma_2\\ s(e_1) = i,\ t(e_2) = j}} t^{f(e_1)+f(e_2)}
\end{align*}

Similarly by induction we obtain for all $k$,

\[\left(M^k\right)_{ij} = \sum_{\substack{p_k \in \Sigma_k\\ s(p_k) = i,\ t(p_k) = j}} t^{S^\sigma_k f (p_k)}. \]

Hence the coefficient of $t^a$ in the polynomial $(M^k)_{ij}$ is
\[\#\left\{p_k\in \Sigma_k:\ s(p_k) = i,\ t(p_k) = j,\ S^\sigma_k f (p_k) = a  \right\} = b^{(k)}_{ij},\]
and thus $M^{(k)} = M^k.$
\end{proof}

Now consider that each {set} $\K_{i,0}$ is a union of floors of many different towers. 
Adding up the measures of each floor, we obtain an expression for 
the measure of $\K_{i,0}$.

\begin{prop}\label{invariant recurrence}
Suppose that $\mu$ is an invariant measure for $T_\phi$. Then for any $k \geq 1$, for any $i$,

\[\mu\left(\K_{i,0}\right) = \sum_{j=1}^d \sum_{a\in G} b^{(k)}_{ij,a}\ \mu\left(J^{(k)}_{j,a}\right).\]
\end{prop}
\begin{proof}
Since $\mu$ is invariant, the measure of any floor of the tower $\widetilde{Z}^{(0,kN)}_{j,a}$ is the same as the measure of the base $J^{(k)}_{j,a}$. 
Adding up over all the towers and counting with multiplicity, which is given by $b^{(k)}_{ij,a}$ by definition, we obtain the result. 
\end{proof}

\subsection{Applying the recurrence to Maharam measures}
Consider now the Maharam measure $\mu_\psi$ for some homomorphism $\psi:G \to \R$. We normalise the measure such that $\mu_\psi(\K\times\{0\}) = 1$.

Now we use the defining property of Maharam measures, which is that for any measurable set $U \subset \K$, 
$\mu_\psi (U \times \{a\}) = e^{\psi(a)} \mu_\psi(U\times\{0\})$.

For the standard basis $\{u_1,\dots u_m\}$ of $G$, let $\lambda_i = e^{\psi(u_i)}$, 
and $\lambda = (\lambda_1,\dots,\lambda_k)$. Then $e^{\psi(a)} = \lambda^a$ for all $a\in G$.

\begin{lemma}
Let $M(\lambda)$ be the matrix obtained by evaluating the entries of $M$ at $t=\lambda$. Let $w^{(k)}$ be the vector with entries 
$w^{(k)}_i = \mu_\psi\left(J^{(k)}_{i,0}\right)$. Then

\[w^{(0)} = M(\lambda)^k w^{(k)}.\]
\end{lemma}
\begin{proof}
From Proposition \ref{invariant recurrence}, Lemma \ref{M is cocycle} and using the scaling property of $\mu_\psi$, we see that

\[w^{(0)}_i = \sum_{j=1}^d \sum_{a\in G} b^{(k)}_{ij,a} \lambda^a\ \mu_\psi\left(J^{(k)}_{j,0}\right) 
= \sum_{j=1}^d \left( M(\lambda)^k\right)_{ij} w^{(k)}_j = \left(M(\lambda)^k w^{(k)}\right)_i.\]

\end{proof}

Now, by the positivity of $A$, $M(\lambda)$ is positive for any $\lambda > 0$, so we can apply the Perron-Frobenius theorem to it. 

Denote by $v_{\lambda}$ the Perron-Frobenius eigenvector of $M(\lambda)$, normalised such that $\sum_{i=1}^d (v_\lambda)_i = 1$,
 and by $r_\lambda$ the corresponding eigenvalue.
 
\begin{thm}
$w^{(k)} = r_\lambda^{-k} v_\lambda$ and hence for any $k$, if $p_k\in \Sigma_k$, then 
\[\scalebox{1.2}{$\displaystyle \mu_\psi\left(J(p_k)\times \{a\}\right) = \frac{\lambda^{a+S^\sigma_k f (p_k)}\, (v_\lambda)_{t(p_k)}}{r_\lambda^k}.$}\]
\end{thm}
\begin{proof}
We have by Perron-Frobenius theorem that $\bigcap_{k\geq 1} M(\lambda)^k \big(\R_+\big) = \R_+ v_\lambda$. 

Hence $w^{(0)}$ is a scalar multiple of $v_\lambda$, and since we chose the normalisations to be consistent, $w^{(0)} = v_\lambda$. Then also 
$w^{(k)} = M(\lambda)^{-k} w^{(0)} = r_\lambda^{-k} v_\lambda$.

Now $J(p_k)\times\{a\}$ is in the tower $\widetilde{Z}^{(0,kN)}_{t(p_k),\,a+S^\sigma_k f(p_k)}$, and hence 
\[\scalebox{1.2}{$\displaystyle \mu_\psi\left(J(p_k)\times \{a\}\right) = \lambda^{a+S^\sigma_k f (p_k)}\, w^{(k)}_{t(p_k)} = \frac{\lambda^{a+S^\sigma_k f (p_k)}\, (v_\lambda)_{t(p_k)}}{r_\lambda^k}.$}\]
\end{proof}

We are grateful to the referee for pointing out the close connection to McMullen's work in \cite{McMullen}, where he considers a 
similar question when studying twisted measured laminations preserved by a pseudo-Anosov. 
The connection to our setting comes from the fact that one can interpret Maharam measures arising from homomorphisms $\psi:G \to \R$ as 
corresponding twisted measured laminations.\footnote{A Maharam measure $\mu_\psi$ for a skew-product $T_\phi$ can be interpreted
as a  twisted measured lamination as follows: after suspending the IET $T$ as 
the first return map to a Poincaré section $I$ of a straight-line flow 
on a translation surface $S$, define for $x \in I$ the homology class $\xi(x)\in H_1(S,\Z)$ as the cycle given by the first return of the flow to 
$I$, with the endpoints of the flow arc joined along $I$. Then 
for some $\alpha\in H^1(S,G)$, $\phi$ can be seen as $\phi(x) = \alpha(\xi(x))$. For a homomorphism $\psi:G \to \R$, a quasi-invariant measure 
$\nu_\psi$ corresponds to a twisted measured lamination in $\mathcal{ML}_s$ for $s=-\psi\circ \alpha \in H^1(S,\R)$. (Here  $\mathcal{ML}_s$ 
are measured laminations twisted by the cohomology class $s$, in the sense that for a $\mu \in  \mathcal{ML}_s$, for any $\gamma\in \pi_1(S)$, 
$\gamma_*\mu = e^{s(\gamma)}\mu$ -- see \cite{McMullen} for details.)}

In \cite[Thm 8.3]{McMullen} McMullen shows that one can construct twisted measured laminations from the 
Perron-Frobenius eigenvectors of some matrices of Laurent polynomials, which correspond in the language of train tracks to the level-counting 
cocycle that we defined using Rauzy-Veech induction. While McMullen's setting only applies to laminations preserved by a 
pseudo-Anosov (corresponding to periodic type IETs), the level-counting cocycle as defined here can be defined analogously for any skew-product.

\def\thethmglobal{2}
\begin{thmglobal}
The measures $\mu_\psi$ depend weak-* continuously on $\psi \in \Hom(G,\R)$, where we put the standard topology on $\Hom(G,\R) \cong \R^m$.
\end{thmglobal}
\begin{proof}
$M(\lambda)$ depends continuously on $\psi$, hence so does $r_\lambda$, and as the Perron-Frobenius eigenvalue is always simple, 
the eigenvector $v_\lambda$ also depends continuously on $\lambda$. Thus for any $p_k\in \Sigma_k$ and $a\in G$,
$\mu_\psi(J(p_k)\times\{a\})$ depends continuously on $\psi$. Increasing $k$, we get arbitrarily fine partitions of $\K\times G$, and hence we deduce that 
for any Borel $U \subset \K\times G$, $\mu_\psi(U)$ is continuous with respect to $\psi$.
\end{proof}

\appendix
\section{Proof of $\mathcal{O}(\tau_{\phi_f}) = \mathfrak{T}(\sigma_f)$}\label{app:orb=tail}
In this appendix we prove that given a map $f:\Sigma \to G$, its tail cocycle $\phi_f$ (see Definition \ref{defn:phi_f}) satisfies the equivalence 
of relations  $\mathcal{O}(\tau_{\phi_f}) = \mathfrak{T}(\sigma_f).$ In fact this is the entire reason for the definition of the tail cocycle.
This proof is an extended version of the discussion at the beginning of \S2 in \cite{ANSS}.
\begin{prop}
Let $\tau:\Sigma^0 \to \Sigma^0$ be an adic map and $\sigma:\Sigma^0\to\Sigma^0$ be the left shift. Let $G$ be a group and $f:\Sigma^0 \to G$ be any 
map. Let $\phi_f:\Sigma^0 \to G$ be as defined in Definition \ref{defn:phi_f}. Then the skew-products $\sigma_f$ and $\tau_{\phi_f}$ satisfy:
\[\mathcal{O}(\tau_{\phi_f}) = \mathfrak{T}(\sigma_f).\]
\end{prop}
\begin{proof}
Since $\tau$ is the adic map, we know that $\mathcal{O}(\tau) = \mathfrak{T}(\sigma)$. 

\bigskip
First let us show that $\mathcal{O}(\tau_{\phi_f}) \subset \mathfrak{T}(\sigma_f)$.\\
Suppose $(p',a') = \tau_{\phi_f}^n(p,a)$ for some $(p,a) \in \Sigma^0\times G$. 
Then $\tau^n p = p'$, so there exists some $N$ such that $\sigma^N(p) = \sigma^N(p')$.

From the definition of $\phi_f$,
\begin{equation}\label{eqn:BS}
\begin{split}
S^\tau_n \phi_f (p) &= \sum_{i=0}^{n-1} \left[ \sum_{k=0}^{N-1} \left(f(\sigma^k \tau^i p) - f(\sigma^k \tau^{i+1} p)\right)  \right]\\
&= \sum_{k=0}^{N-1} \left(f(\sigma^k p) - f(\sigma^k \tau^n p)\right)\\
&= S^\sigma_N f(p) - S^\sigma_N f (p').
\end{split}
\end{equation}
Hence
\[a' = a + S^\sigma_N f(p) - S^\sigma_N f (p'),\]
and we see that
\[\sigma_f^N(p,a) = \big(\sigma^N p, a +  S^\sigma_N f(p)\big) = \big(\sigma^N p', a' +  S^\sigma_N f(p')\big) = \sigma_f^N(p',a').\]

\bigskip
Now let us show the converse inclusion, $\mathcal{O}(\tau_{\phi_f}) \supset \mathfrak{T}(\sigma_f)$.\\
Suppose that $(p,a),\, (p',a') \in \Sigma^0\times G$ are such that for some $N>0$, $\sigma_f^N(p,a) = \sigma_f^N(p',a')$. 

Since $\sigma^N p = \sigma^N p'$, there exists some $n\in \Z$ such that $p' = \tau^n p$. WLOG $n>0$ (else change the roles of $p,p'$.) 

From (\ref{eqn:BS}), $S^\tau_n \phi_f (p)= S^\sigma_N f(p) - S^\sigma_N f (p')$. 
Hence 
\[a' = a + S^\sigma_N f(p) - S^\sigma_N f (p') = a+S^\tau_n \phi_f (p),\]
and so 
\[\tau_{\phi_f}^n(p,a) = (p',a').\]
\end{proof}

\section{Proof of $\phi_f = \phi$}\label{app:phi=phi_f}
Here we prove that in the case of a periodic type skew-product $T_\phi$ coded by the skew-product $\tau_\phi$ over the adic map, if $f$ is as defined 
in Definition \ref{defn:f} and $\phi_f$ is the corresponding tail cocycle, then in fact $\phi_f = \phi$.

\begin{defn}
Say a $p = ((j_1,l_1),\dots)\in \Sigma$ is \textbf{top-floor} if $l_1 = q_{j_1}-1$ (i.e the edge $(j_1,l_1)$ is maximal).
This is equivalent to $\pi(p)$ being on the top floor of the tower $\1Z_{j_1}$, hence the terminology. Similarly call $p$ \textbf{bottom-floor} if 
$l_1=0$.
\end{defn}

\begin{lemma}\label{lem:top-floor}
If $p$ is top-floor, then $\sigma\tau p = \tau\sigma p$.
\end{lemma}
\begin{proof}
Suppose $p = (e_1,e_2,\dots)$, with $e_i=(j_i,l_i)$, is top-floor. Let $n\geq 2$ be minimal such that $e_n$ is not maximal. 
Then $\tau p = (f_1,\dots,f_{n-1},f_n,e_{n+1},\dots)$, where $f_n = (j_n,l_n+1)$ and $f_k = (s(f_{k+1}),0)$ for $k<n$. Hence 
\[\sigma \tau p = (f_2,\dots,f_{n-1},f_n,e_{n+1},\dots).\]
On the other hand, $e_n$ is also the first non-maximal edge in $\sigma p = (e_2,e_3,\dots)$, so also 
\[\tau \sigma p = (f_2,\dots,f_{n-1},f_n,e_{n+1},\dots).\]
\end{proof}

\begin{lemma}
	For any $f$, $\phi_f$ satisfies:
	
	\begin{equation}\label{rec_phi_f}
	f(\tau p) = 
	\begin{cases}
		f(p) - \phi_f(p) + \phi_f(\sigma p) & \text{ if } p \text{ is top-floor}\\
		f(p) - \phi_f(p) & \text{ otherwise.}\\ 
	\end{cases}
	\end{equation}
\end{lemma}
\begin{proof}
	For non top-floor $p$, for any $i > 0, \sigma^i(p) = \sigma^i(\tau p)$, hence by definition $\phi_f(p) = f(p) - f(\tau p)$.
	
	For top-floor $p$, we have by Lemma \ref{lem:top-floor}
	 $\tau \sigma p = \sigma \tau p$. Hence
	\begin{align*}
		\phi_f(p) = &f(p) - f(\tau p) + \sum_{i=1}^{\infty} \big(f(\sigma^i p ) - f(\sigma^i \tau p )\big)\\
		=& f(p) - f(\tau p) + \sum_{i=1}^{\infty} \big(f(\sigma^{i-1} \sigma p ) - f(\sigma^{i-1} \sigma \tau p )\big)\\
		=& f(p) - f(\tau p) + \sum_{i=0}^{\infty} \big(f(\sigma^i \sigma p ) - f(\sigma^i \tau \sigma p )\big)\\
		=& f(p) - f(\tau p) + \phi_f(\sigma p).\qedhere
	\end{align*}
\end{proof}

Also, from the definition of $f$ via $\phi$ we know that the following recurrence holds.
\begin{lemma}
    \begin{equation}\label{rec_phi}
	f(\tau p) = 
	\begin{cases}
		0 & \text{ if } p \text{ is top-floor}\\
		f(p) - \phi(p) & \text{ otherwise.}\\ 
	\end{cases}
	\end{equation}
\end{lemma}
\begin{proof}
Indeed, if $p$ is top-floor, then $\tau p$ is bottom-floor, so $f(\tau p) = 0$. Otherwise, if $p = ((j_1,l_1),\dots)$,

\begin{align*}
f(\tau p) - f(p) &= \Big[ \sum_{l=0}^{l_1}  -\phi_{s((j_1,l))} \Big] - 
\Big[ \sum_{l=0}^{l_1-1}  -\phi_{s((j_1,l))} \Big] \\
&= -\phi_{s((j_1,l_1))}\\
&= -\phi(p).
\end{align*}
\end{proof}

We'll need one more lemma to put the recurrence relations  (\ref{rec_phi_f}) and (\ref{rec_phi}) in the same form. 
This is the place where we use the assumption $\phi^{(N)} = \phi$. 
(Recall that $\phi^{(N)} = A^T \phi$, where  $A = A^{(0,N)}$.)
\begin{lemma}
If $\phi^{(N)} = \phi$, then for top-floor $p$, 
\begin{equation}\label{rec_top}
f(p)  - \phi(p) + \phi(\sigma p) = 0.
\end{equation}

Hence
\begin{equation}\label{rec_phi top}
	f(\tau p) = 
	\begin{cases}
		 f(p)  - \phi(p) + \phi(\sigma p) & \text{ if } p \text{ is top-floor}\\
		f(p) - \phi(p) & \text{ otherwise.}\\ 
	\end{cases}
	\end{equation}

\end{lemma}
\begin{proof}
Suppose $p = ((j_1,l_1),e_2,\dots)$ is top-floor. Then $l_1 = q-1$, where $q = q_{j_1}$.  Then from the definition of $f$:

\begin{align*}
f(p) &= - S^T_{q-1} \phi \Big|_{\1J_{j_1}}\\
& = \phi(p) - S^T_q \phi \Big|_{\1J_{j_1}}.
\end{align*}

{Since $q = q_{j_1}^{(0,N)}$, by Lemma \ref{lem:BS}, $S^T_q \phi \Big|_{\1J_{j_1}} = \phi^{(N)}_{j_1}$.}
Hence
\[f(p) = \phi(p) - {\phi}^{(N)}_{j_1}.\]

Now the first edge of $\sigma p$ is $e_2$, so $\phi(\sigma p) = \phi_{s(e_2)}$. As  $s(e_2) = j_1$, $\phi(\sigma p) = \phi_{j_1}$. 
By assumption $\phi^{(N)} = \phi$, hence $f(p) = \phi(p)-\phi(\sigma p)$.

For the last equality, substitute the relation we just obtained into (\ref{rec_phi}).
\end{proof}

Now we combine these lemmas.
\begin{prop}\label{phi_f = phi}
If $\phi^{(N)} = \phi$, then $\phi_f(p) = \phi(p)$ for all $p\in \Sigma\setminus\Sigma_{max}$.
\end{prop}
\begin{proof}
By comparing equations (\ref{rec_phi_f}) and (\ref{rec_phi top}), we see that for all non top-floor $p$, $\phi_f(p) = \phi(p)$. 

For top-floor $p$ we deduce that
\[\phi(p) - \phi(\sigma p) = \phi_f(p) - \phi_f(\sigma p).\]
Since $p\notin \Sigma_{max}$, there exists some minimal $k$ for which $\sigma^k p$ is not 
top-floor. Then $\phi(\sigma^k p) = \phi_f(\sigma^k p)$, and so $\phi(p) = \phi_f(p)$.
\end{proof}

\section{Proof that the subset $\Delta \subset G$ is a subgroup.}\label{app:Delta}
Here we give a proof of the statement that $\Delta \subset G$ (defined below) is a subgroup. This is a well-known fact that has been 
stated by many authors, for example Marcus and Tuncel in \cite{Marcus_Tuncel_1991}, but since we did not find a proof in the literature, we 
present one here for completeness.

\begin{prop}
Let $\sigma: \Sigma \to \Sigma$ be a subshift of finite type with positive transition matrix $A$. Let $G$ be an abelian group and $f:\sigma \to G$ a function. 
Let $\textrm{Fix}_n = \{p\in \Sigma: \sigma^n(p) = p\}$. 
Then the set
\[\Delta := \{S^\sigma_n f (p_1) - S^\sigma_n f (p_2): p_1,p_2 \in \textrm{Fix}_n(\sigma) \text{ for some } n\geq 1\}\]
is a subgroup of $G$.
\end{prop}
\begin{proof}
Since $-(S^\sigma_n f (p_1) - S^\sigma_n f (p_2)) = S^\sigma_n f (p_2) - S^\sigma_n f (p_1)$, it is immediate that $-\Delta = \Delta$.

Thus what we are required to show is that for any two elements $\delta_p,\delta_q \in \Delta$, $\delta_p + \delta_q \in \Delta$.\\

For some $n,m$, $p_1,p_2 \in \textrm{Fix}_n(\sigma)$ and $q_1,q_2 \in \textrm{Fix}_m(\sigma)$  we can write 
$\delta_p =  S^\sigma_n f (p_1) - S^\sigma_n f (p_2)$ and $\delta_q =  S^\sigma_m f (q_1) - S^\sigma_m f (q_2)$.

Denote by $\bar p_i \in \Sigma_n, \bar q_i \in \Sigma_m$ the truncations of $p_i$ and $q_i$ which are the repeating sections of $p_i$ and $q_i$.  
Let $v_i := s(\bar p_i)=t(\bar p_i),\, w_i := s(\bar q_i)=t(\bar q_i)$.

Fix some vertex $u$ and find finite admissible paths $\gamma_{v_i}^{\pm}, \gamma_{w_i}^{\pm}$ such that for $x = v_1,v_2,w_1,w_2$, the path $\gamma^+_x$ 
goes from $u$ to $x$ and $\gamma^-_x$ from $x$ to $u$:
\begin{itemize}
	\item $s(\gamma_{x}^{+}) = u, t(\gamma_{x}^{+}) = x$,
	\item $s(\gamma_{x}^{-}) = x, t(\gamma_{x}^{-}) = u$.
\end{itemize}

See Figure \ref{fig:loops} for illustration. These paths exist due to the positivity of $A$.

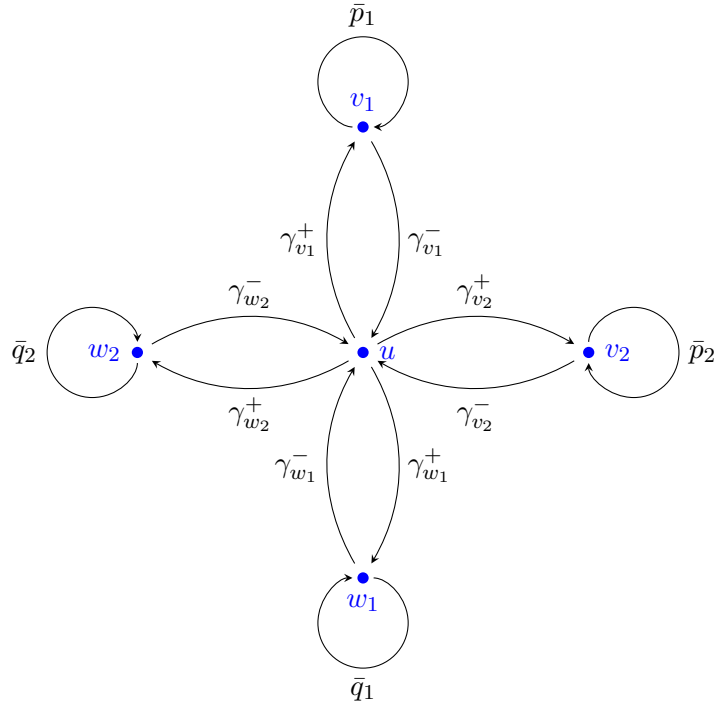
\begin{figure}[H]
\centering
\begin{tikzpicture}[scale=3, vertex/.style={circle,inner sep=1.5pt,fill=blue}, every label/.style={text=blue}, bend left,->,>=stealth,shorten > = 4pt,shorten < = 4pt]
	\node[vertex,label=right:{$u$}] at (0,0) (0) {};
	\node[vertex,label=right:{$v_2$}] at (1,0)  (2) {};
	\node[vertex,label=above:{$v_1$}] at (0,1)  (1) {};
	\node[vertex,label=left:{$w_2$}] at (-1,0) (4) {};
	\node[vertex,label=below:{$w_1$}] at (0,-1) (3) {};
	\path (0) edge node[left]{$\gamma^+_{v_1}$} (1);
	\path (1) edge node[right]{$\gamma^-_{v_1}$} (0);
	\path (0) edge node[above]{$\gamma^+_{v_2}$} (2);
	\path (2) edge node[below]{$\gamma^-_{v_2}$} (0);
	\path (0) edge node[right]{$\gamma^+_{w_1}$} (3);
	\path (3) edge node[left]{$\gamma^-_{w_1}$} (0);
	\path (0) edge node[below]{$\gamma^+_{w_2}$} (4);
	\path (4) edge node[above]{$\gamma^-_{w_2}$} (0);
	\draw[shorten > = 4pt,shorten < = 4pt] (1) arc (270:-90:0.2) node[midway,above] {$\bar p_1$};
	\draw[shorten > = 4pt,shorten < = 4pt] (2) arc (180:-180:0.2) node[midway,right] {$\bar p_2$};
	\draw[shorten > = 4pt,shorten < = 4pt] (3) arc (90:-270:0.2) node[midway,below] {$\bar q_1$};
	\draw[shorten > = 4pt,shorten < = 4pt] (4) arc (0:-360:0.2) node[midway,left] {$\bar q_2$};

\end{tikzpicture}
\caption{The paths $\bar p_i, \bar q_i, \gamma^{\pm}_{v_i}, \gamma^{\pm}_{w_i}$ in the graph corresponding to the subshift $\sigma$.}
\label{fig:loops}
\end{figure}

Now define the following paths (where $\cdot$ denotes concatenation):
\begin{itemize}

\item $\bar a_1 = (\gamma^+_{v_1} \cdot \bar p_1 \cdot \gamma^-_{v_1}) \cdot (\gamma^+_{v_2}  \cdot \gamma^-_{v_2})
\cdot (\gamma^+_{w_1} \cdot \bar q_1 \cdot \gamma^-_{w_1}) \cdot (\gamma^+_{w_2} \cdot \gamma^-_{w_2})$,

\item $\bar a_2 = (\gamma^+_{v_1}  \cdot \gamma^-_{v_1}) \cdot (\gamma^+_{v_2} \cdot \bar p_2 \cdot \gamma^-_{v_2})
\cdot (\gamma^+_{w_1} \cdot \gamma^-_{w_1}) \cdot (\gamma^+_{w_2} \cdot \bar q_2 \cdot \gamma^-_{w_2})$.
\end{itemize}

\begin{figure}[H]
\begin{subfigure}[b]{0.4\textwidth}
\centering
\begin{tikzpicture}[scale=1.5, vertex/.style={circle,inner sep=1.5pt,fill=blue}, every label/.style={text=blue}, bend left,->,>=stealth,shorten > = 4pt,shorten < = 4pt]
	\node[vertex] at (0,0) (0) {};
	\node[vertex] at (1,0)  (2) {};
	\node[vertex] at (0,1)  (1) {};
	\node[vertex] at (-1,0) (4) {};
	\node[vertex] at (0,-1) (3) {};
	\path (0) edge[red]  (1);
	\path (1) edge[red]  (0);
	\path (0) edge[red]  (2);
	\path (2) edge[red]  (0);
	\path (0) edge[red]  (3);
	\path (3) edge[red]  (0);
	\path (0) edge[red]  (4);
	\path (4) edge[red]  (0);
	\draw[shorten > = 3pt,shorten < = 3pt,red] (1) arc (270:-90:0.2) ;
	\draw[shorten > = 3pt,shorten < = 3pt] (2) arc (180:-180:0.2) ;
	\draw[shorten > = 3pt,shorten < = 3pt,red] (3) arc (90:-270:0.2) ;
	\draw[shorten > = 3pt,shorten < = 3pt] (4) arc (0:-360:0.2) ;
	
	\node[red] at (0.7,0.7) {$\bar a_1$};

\end{tikzpicture}
\end{subfigure}
\begin{subfigure}[b]{0.4\textwidth}
\centering
\begin{tikzpicture}[scale=1.5, vertex/.style={circle,inner sep=1.5pt,fill=blue}, every label/.style={text=blue}, bend left,->,>=stealth,shorten > = 4pt,shorten < = 4pt]
	\node[vertex] at (0,0) (0) {};
	\node[vertex] at (1,0)  (2) {};
	\node[vertex] at (0,1)  (1) {};
	\node[vertex] at (-1,0) (4) {};
	\node[vertex] at (0,-1) (3) {};
	\path (0) edge[orange]  (1);
	\path (1) edge[orange]  (0);
	\path (0) edge[orange]  (2);
	\path (2) edge[orange]  (0);
	\path (0) edge[orange]  (3);
	\path (3) edge[orange]  (0);
	\path (0) edge[orange]  (4);
	\path (4) edge[orange]  (0);
	\draw[shorten > = 3pt,shorten < = 3pt] (1) arc (270:-90:0.2) ;
	\draw[shorten > = 3pt,shorten < = 3pt,orange] (2) arc (180:-180:0.2) ;
	\draw[shorten > = 3pt,shorten < = 3pt] (3) arc (90:-270:0.2) ;
	\draw[shorten > = 3pt,shorten < = 3pt,orange] (4) arc (0:-360:0.2) ;
	
	\node[orange] at (0.7,0.7) {$\bar a_2$};

\end{tikzpicture}
\end{subfigure}
\caption{The paths $\bar a_1$ and $\bar a_2$.}
\end{figure}
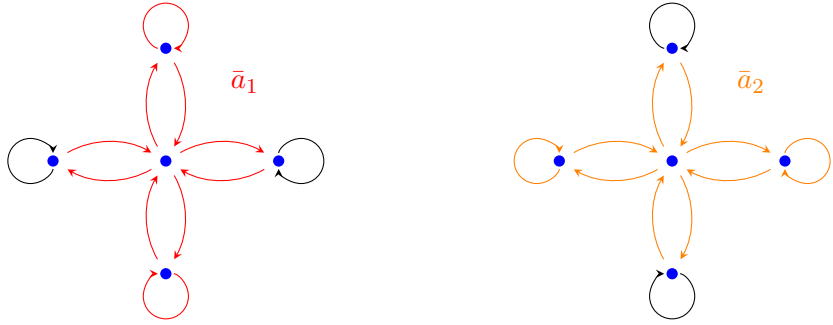

Let also $\bar b = (\gamma^+_{v_1}  \cdot \gamma^-_{v_1}) \cdot (\gamma^+_{v_2}  \cdot \gamma^-_{v_2})
\cdot (\gamma^+_{w_1}  \cdot \gamma^-_{w_1}) \cdot (\gamma^+_{w_2} \cdot \gamma^-_{w_2})$.
Define the infinite periodic paths $a_i$ and $b$ by repeating $\bar a_i$ (respectively $\bar b$) periodically.
If $\bar b$ has length $k$, then $\bar a_i$ both have lengths 
$N = k+n+m$, and so $a_1,a_2$ both lie in $\textrm{Fix}_N(\sigma)$.

Now, by construction, 
\begin{align*}
S^\sigma_N f(a_1) - S^\sigma_N f(a_2) &= S^\sigma_k (b) + S^\sigma_n f(p_1) + S^\sigma_m f(q_1) - 
\big(S^\sigma_k (b) + S^\sigma_n f(p_2) + S^\sigma_m f(q_2)\big) \\
&= S^\sigma_n f(p_1) - S^\sigma_n f(p_2) + S^\sigma_m f(q_1) - S^\sigma_m f(q_2) = \delta_p + \delta_q.
\end{align*}
Hence we deduce that $\delta_p+\delta_q \in \Delta$ and thus $\Delta$ is a group.
\end{proof}

\bibliography{periodic2}{}
\bibliographystyle{amsalpha}
\end{document}